\documentclass[10pt]{amsart}

\usepackage{mathrsfs,txfonts,amssymb}
\input xy
\xyoption{all}

\newcommand{\RR}{\mathbb{R}}
\newcommand{\OO}{\mathscr{O}}

\newcommand{\Cfield}{\mathbb{C}}

\newcommand{\ep}{\varepsilon}
\newcommand{\Spec}{\textnormal{Spec}\,}

\newcommand{\image}{\textnormal{im}\,}
\newcommand{\kernel}{\textnormal{ker}\,}
\newcommand{\cokernel}{\textnormal{coker}\,}
\newcommand{\degree}{\textnormal{deg}\,}

\newcommand{\Hom}{\textnormal{Hom}}
\newcommand{\dimension}{\textnormal{dim}\,}

\newcommand{\rank}{\textnormal{rank}\,}

\newcommand{\Ext}{\textnormal{Ext}}

\newcommand{\al}{\alpha}
\newcommand{\cone}{\textnormal{cone}}

\newcommand{\Coh}{\textnormal{Coh}}

\newcommand{\Ap}{\mathcal{A}^p}

\newcommand{\arinj}{\ar@{^{(}->}}
\newcommand{\arsurj}{\ar@{->>}}
\newcommand{\areq}{\ar@{=}}
\newcommand{\iotam}{{\iota_m}}
\newcommand{\Lo}{\overset{L}{\otimes}}
\newcommand{\HH}{\mathcal{H}}

\newcommand{\wt}{\widetilde}

\newtheorem{theorem}{Theorem}[section]
\newtheorem{lemma}[theorem]{Lemma}
\newtheorem{coro}[theorem]{Corollary}
\newtheorem{pro}[theorem]{Proposition}

\newtheorem{definition}[theorem]{Definition}

\begin{document}

    \title{Moduli of PT-Semistable Objects I}
    \author{Jason Lo}
\begin{abstract}
        We show boundedness for PT-semistable objects of any Chern classes on a smooth projective three-fold $X$.  Then we show that the stack of objects in the heart $\langle \Coh_{\leq 1}(X), \Coh_{\geq 2}(X)[1] \rangle$ satisfies a version of the valuative criterion for completeness.  In the remainder of the paper, we give a series of results on how to compute cohomology with respect to this heart.
\end{abstract}

\address{Department of Mathematics, University of Missouri at Columbia, USA}
\email{locc@missouri.edu}

\maketitle

\section{Introduction}

This is the first of two papers, which study PT-semistable complexes in the derived category and their moduli.

The first main result of this paper is Proposition \ref{pro-boundedness}, a boundedness result for PT-semistable objects.  This would imply that the moduli of PT-semistable objects, once constructed, is of finite-type.

The second main result of this paper is Theorem \ref{theorem-completenessofheart-3}, which shows that objects in the heart $\Ap = \langle \Coh_{\leq 1}(X), \Coh_{\geq 2}(X)[1] \rangle$ satisfy the valuative criterion for completeness when $X$ is a three-fold.   This theorem, and the other  results on computing cohomology with respect to the t-structure whose heart is $\Ap$, lay the groundwork for performing semistable reduction for flat families of objects in $\Ap$, a technique that we will generalise from the setting of sheaves to the setting of derived category in the sequel to this paper.  Actual construction of the moduli spaces of PT-semistable objects will occur in the sequel to this paper.

Even though the arguments in this paper are only written down for PT-stability and  the heart $\Ap$, they could be formalised further, and should also work for a wider class of stability conditions and t-structures.

\subsection{Notation}

Throughout this paper, $k$ will be an algebraically closed field of characteristic 0.  And $R$ will denote a discrete valuation ring (DVR), not necessarily complete, with uniformiser $\pi$ and field of fractions $K$.  Unless specified, $X$ will always denote a smooth projective three-fold over $k$.

We will write $X_R := X \otimes_k R$, and $X_K := X \otimes_R K$.  For any integer $m \geq 1$, let $X_m := X \otimes_k R/\pi^m$, and let
$$ \iota_m : X_m \hookrightarrow X_R$$
denote the closed immersion.  We will often write $\iota$ for $\iota_1$, and $X_k$ for the central fibre of $X_R$.  For integers $1 \leq m' < m$, let
$$\iota_{m,m'} :  X_{m'} \hookrightarrow X_m$$ denote the closed immersion.  We also write
$$ j : X_K \hookrightarrow X_R$$
for the open immersion.

Note that the pushforward functor $\iota_\ast : \Coh (X_k) \to \Coh (X_R)$ is exact, while the pullback $\iota^\ast : \Coh (X_R) \to \Coh (X_k)$ is  right-exact.  Similarly for the pushforward ${\iota_{m,m'}}_\ast$.  The pullback $j^\ast : \Coh (X_R) \to \Coh (X_K)$ is exact.

For a Noetherian scheme $Y$, we will always write $\mbox{Kom} (Y)$ for the category of chain complexes of coherent sheaves on $Y$, $D^b(Y)$ for the bounded derived category of coherent sheaves, and $D(Y)$ for the unbounded derived category of coherent sheaves.

For $m \geq 1$, the subcategories $\Coh_{\leq 1}(X_m)$ and $\Coh_{\geq 2}(X_m)$ form a torsion pair in $\Coh (X_m)$, and so tilting gives us the heart of a t-structure \[\Ap_m := \Ap (X_m) := \langle \Coh_{\leq 1}(X_m), \Coh_{\geq 2}(X_m)[1]\rangle\] on $D^b(X_m)$.  In fact, this also defines a t-structure on $D(X_m)$ (see Proposition \ref{tstructureDX}).  The truncation functors associated to this t-structure will be denoted by $\tau^{\leq 0}_{\Ap_m}, \tau^{\geq 0}_{\Ap_m}$, and the cohomology functors denoted by $\HH^i_{\Ap_m}$.  We will drop the subscripts when the context is clear.  On any Noetherian scheme $Y$, the cohomology functors with respect to the standard t-structure on $D(Y)$ will always be denoted by $H^i$.  On $X_K$, let $\Ap_K (X_K)$ or $\Ap_K$ denote the heart $\langle \Coh_{\leq 1}(X_K), \Coh_{\geq 2}(X_K)[1]\rangle$.

We will use $D^{\leq 0}_{\Ap_m}, D^{\geq 0}_{\Ap_m}$ to denote the full subcategories of $D(X_m)$
\begin{align*}
  D^{\leq 0}_{\Ap_m} &= \{ E \in D(X_m) : \HH^i(E)=0 \text{ for all } i >0 \} \\
  &= \{ E \in D(X_m) : H^i(E)=0 \text{ for all } i >1, H^0(E) \in \Coh_{\leq 1}(X_m) \}, \\
  D^{\geq 0}_{\Ap_m} &= \{ E \in D(X_m) : \HH^i(E)=0 \text{ for all } i<0 \} \\
  &= \{ E \in D(X_m) : H^i(E)=0 \text{ for all }i<-2, H^{-1}(E) \in \Coh_{\geq 2}(X_m) \}.
\end{align*}

In summary, we have the following maps between the various schemes:
\begin{equation*}
  \iota_{m,m'} : X_{m'} \hookrightarrow X_m, \quad  \iota_m : X_m \to X_R,\quad  j : X_K \hookrightarrow X_R
\end{equation*}
and associated pullback and pushforward functors
\begin{align*}
  \iota_{m,m'}^\ast &: D(X_m) \to D(X_{m'}), \quad {\iota_{m,m'}}_\ast : D(X_{m'}) \to D(X_m) \\
  \iota_m^\ast &: D(X_R) \to D(X_m), \quad \iotam_\ast : D(X_m) \to D(X_R) \\
  j^\ast &: D(X_R) \to D(X_K).
\end{align*}
Since ${\iota_{m,m'}}_\ast$ is exact, it takes $\Ap_{m'}$ into $\Ap_m$.

Since $\iota_m$ is a closed immersion, it is a projective morphism.  Hence we have the adjoint pair $L\iota_m^\ast \dashv \iotam_\ast$, i.e.\ $L\iota_m^\ast$ is the left adjoint, and $\iotam_\ast$ the right adjoint \cite[p.83]{FMTAG}.  Similarly, we have the adjoint pair $ L\iota_{m,m'}^\ast \dashv {\iota_{m,m'}}_\ast$ for any $1\leq m' < m$.

Consistent with the definitions introduced in \cite{AP} and \cite{BSMSKTS}, we will use the following notion of flatness for derived objects:

\begin{definition}
Let $S$ be a Noetherian scheme over $k$, and $X$ a smooth projective three-fold over $k$.  We say an object $E \in D^b(X \times S)$ is a flat family of objects in $\Ap$ over $S$ if, for all closed points $s \in S$, we have
\[
E|_s := L\iota_s^\ast E \in \Ap (X|_s) = \langle \Coh_{\leq 1}(X|_s), \Coh_{\geq 2}(X|_s)[1]\rangle
\]
where $\iota_s : X|_{s} \hookrightarrow X \times S$ is the closed immersion of the fibre over $s$.
\end{definition}


Finally, given a complex $E^\bullet \in D^b(X)$, we say $E^\bullet$ is of dimension $d$ if the dimension of the support of $E^\bullet$, defined to be the union of the supports of the various cohomology $H^i(E^\bullet)$, is $d$.

\subsection{Acknowledgements} This paper and its sequel grew out of my doctoral thesis  at Stanford University.  And so, first and foremost, I would like to thank my thesis advisor, Jun Li, for his patience and constant encouragement while I worked on this project.  I would also like to thank Arend Bayer for generously providing comments and suggestions as this work took shape.  For the many useful conversations, I thank Brian Conrad, Young-Hoon Kiem, Ravi Vakil and Ziyu Zhang.  And for graciously answering my questions at various points in time, I thank Dan Edidin, Max Lieblich, Amnon Neeman, Alexander Polishchuk and Yukinobu Toda.

\section{Background}

This section is for the readers who wish to be reminded of the basics of t-structures and polynomial stability.  We also define PT-stability in this section.

\subsection{$\mu$-Stability}

Let $(X,H)$ be a smooth projective variety of dimension $n$.  For any coherent sheaf $F$ on $X$, we  define its degree (with respect to $H$) as $\degree (F) = \int_X c_1(F)\cdot c_1(H)^{n-1}$.   Recall that we define the slope function $\mu$ by $\mu (F) = \frac{\degree (F)}{\rank (F)}$.

If $F$ is a torsion sheaf, we simply define the value to be $+\infty$. For a nonzero torsion-free sheaf $F$ on $X$, we say it is \textit{$\mu$-stable} (resp.\ \textit{$\mu$-semistable}) if, for all nonzero subsheaves $G \subset F$ with $\rank (G)< \rank (F)$, we have $\mu (G) < \mu (F)$ (resp.\ $\mu (G) \leq \mu (F)$).  Given any coherent sheaf $F$ on $X$, there exists a unique filtration, called the \textit{Harder-Narasimhan (HN) filtration}, by subsheaves
\[
 tors(F)=F_0 \subset F_1 \subset \cdots \subset F_m = F,
 \]
where $F_0=tors(F)$ is the torsion subsheaf of $F$, and the $F_i/F_{i-1}$ are all torsion-free $\mu$-semistable sheaves with strictly decreasing slopes:
\[
 \mu (F_1/F_0) > \mu (F_2/F_1) > \cdots > \mu (F_m/F_{m-1}).
 \]

\subsection{$t$-Structures}\label{section-tstructures}

Having a $t$-structure on a triangulated category allows us to compute cohomology.  Recall:

\begin{definition}\cite{GM}
Let $\mathcal{D}$ be a triangulated category.  A t-structure on $\mathcal{D}$ is a pair of strictly full subcategories $(\mathcal{D}^{\leq 0}, \mathcal{D}^{\geq 0})$ satisfying the following conditions, where we define $\mathcal{D}^{\leq n} := \mathcal{D}^{\leq 0}[-n]$ and $\mathcal{D}^{\geq n} := \mathcal{D}^{\geq 0}[-n]$:
\begin{enumerate}
\item $\mathcal{D}^{\leq 0} \subset \mathcal{D}^{\leq 1}$ and $\mathcal{D}^{\geq 0}\subset \mathcal{D}^{\geq 1}$.
\item $\Hom (X,Y)=0$ for $X \in \mathcal{D}^{\leq 0}$, $Y \in\mathcal{D}^{\geq 1}$.
\item For any $X \in \mathcal{D}$ there exists a distinguished triangle $A \to X \to B \to A[1]$ with $A \in \mathcal{D}^{\leq 0}$, $B \in \mathcal{D}^{\geq 1}$.
\end{enumerate}
The heart of the t-structure is the full subcategory $\mathcal{A} = \mathcal{D}^{\geq 0} \cap \mathcal{D}^{\leq 0}$.
\end{definition}

The heart of a t-structure is always an abelian category \cite[Theorem IV-$\S$4.4]{GM}.  We say that we have a short exact sequence $0 \to A \to B \to C \to 0$ in $\mathcal{A}$ if and only if we have an exact triangle $A \to B \to C \to A[1]$ in $\mathcal{D}$, and all the objects $A,B,C$ lie in $\mathcal{A}$.  For instance, a morphism $f : A \to B$ in $\mathcal{D}$ is an injection in $\mathcal{A}$ if $A, B$ and $\cone (f)$ are all objects in $\mathcal{A}$.

 For any Noetherian scheme $X$, the bounded derived category $D^b(X)$ of coherent sheaves on $X$ is a triangulated category, with the standard t-structure given by
\begin{align*}
  D^{\leq 0} &= \{ E \in D^b(X) : H^i(E)=0 \text{ for all } i >0 \} \\
  D^{\geq 0} &= \{ E \in D^b(X) : H^i(E)=0 \text{ for all } i <0 \}.
\end{align*}
The heart of the standard t-structure is $D^{\leq 0} \cap D^{\geq 0} = \Coh (X)$.

\noindent\textit{Truncation functors.} Whenever we have a t-structure $(\mathcal{D}^{\leq 0}, \mathcal{D}^{\geq 0})$ on a triangulated category $\mathcal{D}$, there are associated truncation functors:
\begin{equation*}
  \tau^{\leq n} : \mathcal{D} \to \mathcal{D}^{\leq n} \text{ and } \tau^{\geq n} : \mathcal{D} \to \mathcal{D}^{\geq n}
\end{equation*}
that are right and left adjoint to the corresponding embedding functors, respectively.  Moreover, for any object $X \in \mathcal{D}$, there is an exact triangle of the form
\[
 \tau^{\leq 0} X \to X \to \tau^{\geq 1} X \to \tau^{\leq 0}X[1]
 \]
and the construction of this exact triangle is functorial.  Whenever we have two exact triangles $X' \to X \to X'' \to X'[1]$ where $X' \in \mathcal{D}^{\leq 0}$ and $X'' \in \mathcal{D}^{\geq 1}$, they must be canonically isomorphic.

 For the standard t-structure $(D^{\leq 0}, D^{\geq 0})$ on the bounded derived category $D^b(X)$ of a Noetherian scheme, and a complex $E^\bullet = [ \cdots \to E^{i-1} \overset{d^{i-1}}{\to} E^i \to \cdots]$ of coherent sheaves on $X$, we have
\begin{align*}
  \tau^{\leq i} E^\bullet &= [\cdots \to E^{i-1} \to \kernel d^i \to 0 \to \cdots ], \\
  \tau^{\geq i} E^\bullet &= [\cdots \to 0 \to E^i/\image d^{i-1} \to E^{i+1} \to \cdots],
\end{align*}
and natural maps in the category of chain complexes $\tau^{\leq i}E^\bullet \to E^\bullet$ and $E^\bullet \to \tau^{\geq i}E^\bullet$.  It is easy to see that $H^i (E^\bullet ) = \tau^{\leq i}\tau^{\geq i} E^\bullet = \tau^{\geq i}\tau^{\leq i}E^\bullet$, and other properties such as $\tau^{\leq m} E^\bullet \cong \tau^{\leq m}\tau^{\leq n}E^\bullet$ if $m \leq n$.

\noindent\textit{Cohomology functors.}  With notation as above, define the cohomology functor
$$ \mathcal{H}^i := \tau^{\leq i}\tau^{\geq i} : \mathcal{D} \to \mathcal{A},$$
which is naturally isomorphic to the functor $\tau^{\geq i}\tau^{\leq i}$. In the case of the standard t-structure on $D^b(X)$, for example, this would just be the usual cohomology functor for chain complexes of coherent sheaves. The subcategories $\mathcal{D}^{\leq n}$ and $\mathcal{D}^{\geq n}$ can be described as
\begin{align*}
  \mathcal{D}^{\leq n} &= \{ X \in \mathcal{D} : \mathcal{H}^i(X) = 0 \text{ for all } i >n \} \\
  \mathcal{D}^{\geq n} &= \{ X \in \mathcal{D} : \mathcal{H}^i(X) = 0 \text{ for all } i <n \}.
\end{align*}

Given a triangulated category $\mathcal{D}$ with t-structure $(\mathcal{D}^{\leq n}, \mathcal{D}^{\geq n})$ with heart $\mathcal{A}$, for any nonzero object $E \in \mathcal{D}$ there are integers $m<n$ and a diagram of the form
\begin{equation*}
\def\objectstyle{\scriptstyle}
\def\labelstyle{\scriptstyle}
\xymatrix@-.5pc@R=1pc{
0=E_m \ar[rr] & & E_{m+1} \ar[dl] \ar[rr] & & E_{m+2} \ar[r] \ar[dl] & \cdots \ar[r] & E_{n-1} \ar[rr] & & E_n = E \ar[dl] \\
 & A_{m+1} \ar[ul]^{[1]} & & A_{m+2} \ar[ul]^{[1]} & & & & A_n \ar[ul]^{[1]} &
}
\end{equation*}
where the triangles are exact, and $A_i[-i] \in \mathcal{A}$ for each $i$.  In fact, $A_i[-i] \cong \mathcal{H}^i(E)$ are the cohomology objects of $E$.  A t-structure is uniquely determined by its heart.

\noindent\textit{Exact sequences.} Whenever we have a triangulated category $\mathcal{D}$ with a t-structure with heart $\mathcal{A}$, any exact triangle
$$ A \to B \to C \to A[1]$$
in $\mathcal{D}$ yields a long exact sequence of cohomology objects
$$ \cdots \to \HH^i(A) \to \HH^i(B) \to \HH^i(C) \to \HH^{i+1}(A) \to \cdots$$
in the heart $\mathcal{A}$.

\noindent\textit{Tilting.}  Given any abelian category $\mathcal{A}$, the process of tilting constructs a non-standard t-structure on the triangulated category $D^b(\mathcal{A})$ (in fact, on $D(\mathcal{A})$ - see Proposition \ref{tstructureDX}).

\begin{definition}[\cite{TACQA}]
Given a pair of full subcategories $(\mathcal{T},\mathcal{F})$ in an abelian category $\mathcal{A}$, we say that $(\mathcal{T},\mathcal{F})$ is a torsion pair in $\mathcal{A}$ if:
\begin{itemize}
\item $\Hom (T,F)=0$ for all $T \in \mathcal{T}$ and $F \in \mathcal{F}$, and
\item for all $X \in \mathcal{A}$, there is a short exact sequence in $\mathcal{A}$
\[
0 \to T \to X \to F \to 0
\]
where $T \in \mathcal{T}$, $F \in \mathcal{F}$.
\end{itemize}
\end{definition}

Given a torsion pair as above, the following full subcategories of $D^b(\mathcal{A})$ define a t-structure on $D^b(\mathcal{A})$ (\cite[Proposition 2.1]{TACQA}):
\begin{align*}
  \mathcal{D}^{\leq 0} &= \{ E \in D^b(\mathcal{A}) : H^0 (E) \in \mathcal{T}, \text{ and } H^i(E)=0 \text{ for all }i>0 \} \\
  \mathcal{D}^{\geq 0} &= \{ E \in D^b(\mathcal{A}) : H^{-1} (E) \in \mathcal{F}, \text{ and }H^i(E)=0 \text{ for all }i<-1\}.
\end{align*}
And the heart of this t-structure is
\begin{align*}
 \mathcal{D}^{\leq 0} \cap \mathcal{D}^{\geq 0} &= \langle \mathcal{T}, \mathcal{F}[1] \rangle \\
 &= \{ E \in
D^b(\mathcal{A}) : H^0(E)\in \mathcal{T}, H^{-1}(E) \in \mathcal{F}, H^i(E)=0 \text{ for }i \neq 0,-1\}
\end{align*}
where $\langle \mathcal{T}, \mathcal{F}[1]\rangle$ denotes the extension-closed subcategory generated by $\mathcal{T}$ and $\mathcal{F}[1]$.

\noindent\textit{Remark.} Given any object $E$ in the heart $\mathcal{D}^{\leq 0} \cap \mathcal{D}^{\geq 0}$ obtained from tilting as above, there is a short exact sequence $$ 0 \to H^{-1}(E)[1] \to E \to H^0(E) \to 0$$ in the abelian category $\mathcal{D}^{\leq 0} \cap \mathcal{D}^{\geq 0}$.


\noindent\textit{Example.} Let $X$ be a Noetherian scheme of dimension $n$, and $E$ any coherent sheaf on $X$.  For any integer $0\leq d < n$, there is a unique short exact sequence in $\Coh (X)$:
\[0 \to T \to E \to F \to 0\]
where $T$ is the maximal subsheaf of dimension at most $d$, and $F$ has no nonzero subsheaf of dimension $d$ or less.  In other words, if we define
\begin{align*}
  \Coh_{\leq d}(X) &= \{ E \in \Coh (X) : \dimension E \leq d \} \\
  \Coh_{\geq d+1}(X) &= \{ E \in \Coh (X) : \Hom_{\Coh (X)} (F,E)=0 \text{ for all } F \in \Coh_{\leq d}(X) \}
\end{align*}
then $(\Coh_{\leq d}(X),\Coh_{\geq d+1}(X))$ is a torsion pair in the abelian category $\Coh (X)$.  For $0 \leq d' < d$, the category $\Coh_{\leq d'}(X)$ is a Serre subcategory of $\Coh_{\leq d}(X)$, so we may form the quotient category $\Coh_{d,d'}(X) := \Coh_{\leq d}(X)/\Coh_{\leq d'} (X)$, which is an abelian category.  For a coherent sheaf $F$ on $X$, we write $p(F)$ for its reduced Hilbert polynomial, and if $F \in \Coh_{\leq d}(X)$, we write $p_{d,d'}(F)$ for its reduced Hilbert polynomial as an element of $\Coh_{\leq d,d'}(X)$ (see \cite[Section 1.6]{HL}).

In particular, when $X$ is a Noetherian scheme of dimension 3, following the notation in \cite{BayerPBSC}, we will always use $\Ap (X)$, or simply $\Ap$, to denote the heart obtained from tilting $\Coh (X)$ with respect to the torsion pair $(\Coh_{\leq 1}(X),\Coh_{\geq 2}(X))$:
\begin{align*}
\Ap := \Ap (X) &:= \langle \Coh_{\leq 1}(X), \Coh_{\geq 2}(X) [1]\rangle \\
&= \{ E \in D^b(X) : H^0(E) \in \Coh_{\leq 1}(X), H^{-1}(E) \in \Coh_{\geq 2}(X), \\
&\hspace{5cm} H^i(E)=0 \text{ for all } i\neq 0, -1\}.
\end{align*}
Various properties of the heart $\Ap$ have been worked out in \cite{TodaLSOp}.  Note that $\Coh_{\leq 1}(X)$ is an abelian subcategory of $\Ap$.

\subsection{Polynomial Stability and PT-Stability}

Polynomial stability was defined on $D^b(X)$ by Bayer for any normal projective variety $X$ \cite[Theorem 3.2.2]{BayerPBSC}.  While the central charge for a Bridgeland stability condition is required to take values in $\Cfield$, the central charge for a polynomial stability condition takes values in the abelian group  $\Cfield [m]$ of polynomials over $\Cfield$.

The particular class of polynomial stability conditions we will concern ourselves with for the rest of the paper consists of the following data, where $X$ is a smooth projective three-fold:
\begin{enumerate}
\item the heart $\Ap = \langle \Coh_{\leq 1}(X), \Coh_{\geq 2}(X)[1]\rangle$, and
\item a group homomorphism (the central charge) $Z : K(X) \to \Cfield [m]$ of the form
$$ Z(E)(m) = \sum_{d=0}^3 \int_X  \rho_d H^d  \cdot ch(E) \cdot U \cdot m^d$$
where
\begin{enumerate}
\item the $\rho_d \in \Cfield$ are nonzero and satisfy $\rho_0, \rho_1 \in \mathbb{H}$, $\rho_2, \rho_3 \in -\mathbb{H}$, and $\phi (-\rho_2) > \phi (\rho_0) > \phi (-\rho_3) > \phi (\rho_1)$ (see Figure \ref{figure-PTstab} below),
\item $H \in \text{Amp}(X)_\RR$ is an ample class, and
\item $U =1+U_1 + U_2 + U_3\in A^\ast (X)_\RR$ where $U_i \in A^i(X)$.
\end{enumerate}
\end{enumerate}

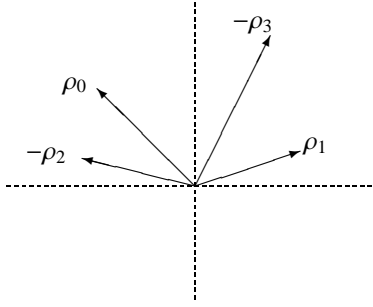
\begin{figure*}[h]
\centering
\setlength{\unitlength}{1mm}
\begin{picture}(50,40)
\multiput(0,15)(1,0){50}{\line(1,0){0.5}}
\multiput(25,0)(0,1){40}{\line(0,1){0.5}}
\put(25,15){\vector(-4,1){15}}
\put(2.5,18.6){$-\rho_2$}
\put(25,15){\vector(-1,1){13}}
\put(7.5,28){$\rho_0$}
\put(25,15){\vector(1,2){10}}
\put(30,36){$-\rho_3$}
\put(25,15){\vector(3,1){14}}
\put(39.5,20){$\rho_1$}
\end{picture}
\caption{Configuration of the $\rho_i$ for PT-stability conditions}
\label{figure-PTstab}
\end{figure*}

The configuration of the $\rho_i$ is compatible with the heart $\Ap$, in the sense that for every nonzero $E \in \Ap$, we have $Z(E)(m) \in \mathbb{H}$ for $m \gg 0$.  So there is a unqiuely determined function $\phi (E)(m)$ (strictly speaking, a uniquely determined function germ) such that
\[
 Z(E)(m) \in \mathbb{R}_{>0} e^{i \pi \phi (E)(m)} \text{ for all } m \gg 0.
 \]
This allows us to define the notion of semistability on objects.  We say that a nonzero object $E$ is \textit{$Z$-semistable} (resp.\ \textit{$Z$-stable}) if for any nonzero subobject $G \hookrightarrow E$ in $\Ap$, we have $\phi (G)(m) \leq \phi (E)(m)$ for $m \gg 0$ (resp.\ $\phi (G)(m) < \phi (E)(m)$ for $m \gg 0$).  We also write $\phi (G) \prec \phi (E)$ (resp.\ $\phi (G) \preceq \phi (E)$) to denote this.  Harder-Narasimhan filtrations for polynomial stability functions exist \cite[Section 7]{BayerPBSC}.

By \cite[Proposition 6.1.1]{BayerPBSC}, with respect to any polynomial stability function from the class above, the stable objects in $\Ap$ with $ch= (-1,0,\beta,n)$ and trivial determinant are exactly the stable pairs in Pandharipande and Thomas' paper \cite{PT}, which are 2-term complexes of the form $$[\OO_X \overset{s}{\to} F]$$ where $F$ is a pure 1-dimensional sheaf and $s$ has 0-dimensional cokernel.  For this reason, and in line with calling a stability function as above a \textit{PT-stability function} in \cite{BayerPBSC}, we call any polynomial stability condition satisfying the above requirements a \textit{PT-stability condition}, and any nonzero object in $\Ap$ semistable (resp.\ stable) with respect to it \textit{PT-semistable} (resp.\ \textit{PT-stable}).

\section{Boundedness of PT-Stable Objects}

From \cite{BayerPBSC}, we know that HN filtrations for polynomial stability conditions exist and are necessarily unique (up to isomorphism).  Let $\phi$ be the phase function of any polynomial stability (not necessarily PT-stability) on some category of perverse coherent sheaves $\Ap$.  Take any $0 \neq E \in \Ap$, and suppose it has HN filtration
\[
0 \neq E_0 \hookrightarrow E_1 \hookrightarrow \cdots \hookrightarrow E_n = E.
\]
Then $E_0$ is a \textit{maximal destabilising subobject} of $E$ in the following sense:

\begin{lemma}\label{lemma-mds-existence}
With the above hypotheses, for any $F \hookrightarrow E$ in $\Ap$ such that $\phi (F) \geq \phi (E_0)$, we have $\phi (F) = \phi (E_0)$ and $F \hookrightarrow E_0$.
\end{lemma}
\begin{proof}
Take any subobject $F$ of $E$ in $\Ap$.  First suppose that $\phi (F) > \phi (E_0)$.  Suppose $F$ has HN filtration $0 \neq F_0 \hookrightarrow F_1 \hookrightarrow \cdots \hookrightarrow F$.  Then the composition $F_0 \hookrightarrow E \twoheadrightarrow E/E_{n-1}$ must be the zero map, because $F_0$ and $E/E_{n-1}$ are both semistable, and $\phi (F_0) \geq \phi (F) > \phi (E_0) \geq \phi (E/E_{n-1})$.  Hence we have $F_0 \hookrightarrow E_{n-1}$ in $\Ap$.  We can continue this process (so we would consider the composition $F_0 \hookrightarrow E_{n-1} \twoheadrightarrow E_{n-1}/E_{n-2}$ next), and eventually obtain $F_0 \hookrightarrow E_0$.  But then $F_0, E_0$ are both semistable, and $\phi (F_0) \geq \phi (F) > \phi (E_0)$, so we have a contradiction.  So if $\phi (F)\geq \phi (E_0)$, equality must hold.

Now suppose $\phi (F) = \phi (E_0)$.  We want to show that $F \hookrightarrow E_0$.  We have $\phi (F_0) \geq \phi (F) = \phi (E_0) > \phi (E_1/E_0) > \phi (E_2/E_1) > \cdots > \phi (E_n/E_{n-1})$, and so by the same arguments as above, we get $F_0 \subseteq E_0$.  This, together with the facts that $F_0, E_0$ are both semistable, and that $\phi (F_0) \geq \phi (F) = \phi (E_0)$, implies we cannot have $\phi (F_0) > \phi (E_0)$.  So $\phi (F_0) = \phi (E_0)$, forcing $\phi (F_0) = \phi (F)$.  This means that $F=F_0$, i.e.\ $F$ itself must be semistable and $F \hookrightarrow E_0$ holds.
\end{proof}

Next, we give a partial characterisation of PT-semistable objects.

\begin{lemma}[Rank-zero PT-semistable objects]\label{lemma-rankzeroPTssobj}
Let $E \in \Ap (X)$ be an object of rank zero.  If $E$ is of dimension 2, then $E$ is PT-semistable if and only if $E=H^{-1}(E)[1]$ where $H^{-1}(E)$ is Gieseker semistable.  If $E$ is of dimension 0 or 1, then $E$ is PT-semistable if and only if $E=H^0(E)$ where $H^0(E)$ is Gieseker semistable.
\end{lemma}

\begin{proof}
When $E$ is of dimension 1 or 0, $E=H^0(E)$, and the result is clear.  Suppose $E$ is of dimension 2.  If $E$ is PT-semistable, then $H^0(E)$ must be zero, or else it would be a destabilising quotient of $E$.  Hence $E=H^{-1}(E)[1]$.  Given any subsheaf $F \subset H^{-1}(E)$ with pure 2-dimensional quotient $H^{-1}(E)/F$, we have $F[1] \hookrightarrow H^{-1}(E)[1]=E$ in $\Ap$, so by the PT-semistability of $E$, $\phi (F[1]) \preceq \phi (E)$, implying the same inequality for the reduced Hilbert polynomials $p(F) \leq p(H^{-1}(E))$.  Hence $H^{-1}(E)$ is Gieseker semistable.

Conversely, suppose $E=H^{-1}(E)[1]$ and $H^{-1}(E)$ is a Gieseker semistable sheaf of pure dimension 2.  Let $B$ be a maximal destabilising subobject of $E$ in $\Ap$ (the existence of $B$ is asserted by Lemma \ref{lemma-mds-existence}).  Then $B$ must be 2-dimensional and PT-semistable.  So $B=H^{-1}(B)$, and $H^{-1}(B)$ is Gieseker semistable, by the previous paragraph.  The cokernel of $B \hookrightarrow E$ must then be $T[1]$ for some pure 2-dimensional sheaf $T$, giving us an injection of coherent sheaves $H^{-1}(B) \subset H^{-1}(E)$.  That $H^{-1}(E)$ is Gieseker semistable means $p(H^{-1}(B)) \leq p(H^{-1}(E))$, which translates to $\phi (B) \preceq \phi (E)$.  Hence $E$ itself is PT-semistable.
\end{proof}

\paragraph{Remark.} Since a rank-zero PT-semistable object is just a Gieseker semistable sheaf (up to shift), the valuative criterion for completeness for rank-zero PT-semistable objects is a classical result (see \cite[Theorem 2.B.1]{HL}, for example).

\begin{lemma}\label{lemma-PTsscharacterise}
Let $X$ be a smooth projective three-fold over $k$, and let $E \in \Ap (X)$.  If $E$ has nonzero rank and is PT-semistable, then $H^0(E)$ is 0-dimensional, and $H^{-1}(E)$ is torsion-free and semistable in $\Coh_{3,1}(X)$; in particular, $H^{-1}(E)$ is $\mu$-semistable.
\end{lemma}

\begin{proof}
If $H^{-1}(E)$ has a torsion subsheaf $T \subset H^{-1}(E)$, then $T$ must be pure 2-dimensional.  Then
\[
0 \to T[1] \to H^{-1}(E)[1] \to (H^{-1}(E)/T)[1] \to 0
\]
 is a short exact sequence in $\Ap$, since all terms are in $\Coh_{\geq 2}(X)[1]$.  So we have an injection $T[1] \hookrightarrow H^{-1}(E)[1]$ in $\Ap$.  On the other hand, we also have the injection $H^{-1}(E)[1] \hookrightarrow E$, so $T[1]$ is a subobject of $E$ in $\Ap$.  However, $T[1]$ is 2-dimensional, and $E$ is 3-dimensional, so by the definition of PT-stability (see Figure \ref{figure-PTstab}), $E$ is unstable, a contradiction.  Therefore $H^{-1}(E)$ must be torsion-free.

On the other hand, $H^0(E)$ must be 0-dimensional, or else $E$ would have a destabilising quotient, namely the surjection $E \twoheadrightarrow H^0(E)$ in $\Ap$.

Now suppose $H^{-1}(E)$ is not semistable in $\Coh_{3,1}(X)$.  Then there is a nonzero proper subsheaf $F \subset H^{-1}(E)$ such that $p_{3,1} (F) > p_{3,1} (H^{-1}(E))$, and the cokernel $H^{-1}(E)/F$ is torsion-free.  Then $F[1] \hookrightarrow H^{-1}(E)[1] \hookrightarrow E$ in $\Ap$, and $F[1]$ destabilises $E$, a contradiction.  Hence $H^{-1}(E)$ is semistable in $\Coh_{3,1}(X)$.
\end{proof}

The following proposition shows that the set of PT-semistable objects of arbitrary, fixed Chern classes is bounded.

\begin{pro}\label{pro-boundedness}
Let $X$ be a smooth projective three-fold, and $\{I_s \in \Ap (X)\}_{s \in S}$ the  set of PT-semistable objects with $ch = (-r,-d,\beta,n)$ indexed by the set $S$.  Then there is a scheme $B$ of finite type over $k$, and a coherent sheaf $F$ on $X \times B$, such that each $I_s$ is quasi-isomorphic to a complex $[E^{-1}_s \to E^0_s]$, where $E^{-1}_s, E^0_s$ occur as fibres of $F$ over closed points of $B$.
\end{pro}

\begin{proof}
When $r=0$, this is a classical result by Lemma \ref{lemma-rankzeroPTssobj}.  Suppose $r \neq 0$.  Fix any $s \in S$.  Note that from the short exact sequence
\begin{equation}\label{es1}
H^{-1}(I_s)[1] \to I_s \to H^0(I_s) \to H^{-1}(I_s)[2],
\end{equation}

we  know
\begin{gather*}
ch_0(H^{-1}(I_s)) = r, ch_1(H^{-1}(I_s)) = d, ch_2(H^{-1}(I_s)) = -\beta ,
\end{gather*}
and since $H^0(I_s)$ is 0-dimensional by Lemma \ref{lemma-PTsscharacterise}, we have $ch_3(H^0(I_s))=\text{length}(H^0(I_s))\geq 0$, and so
\[
ch_3(H^{-1}(I_s)) \geq -n.
\]
Now, by \cite[Theorem 4.8]{MaruBFTFS}, we know that the set $\{H^{-1}(I_s)\}_{s \in S}$ is bounded.

Hence the set $\{H^{-1}(I_s)^\vee \}_{s\in \mathcal{S}}$ of duals of all the $H^{-1}(I_s)$ is also bounded.  Since the $H^{-1}(I_s)^\vee$ are $\mu$-semistable, we can find fixed $\rho, N$ such that there is a surjection
\[
  E := \OO (-\rho)^{\oplus N} \twoheadrightarrow H^{-1}(I_s)^\vee
\]
for any $s \in S$.  This induces
\[
 H^{-1}(I_s) \hookrightarrow H^{-1}(I_s)^{\vee\vee} \hookrightarrow E^\vee,
\]
hence an exact sequence of coherent sheaves
\begin{equation}\label{es3}
0 \to H^{-1}(I_s) \to E^\vee \to R_s \to 0
\end{equation}
where $R_s$, depending on $s$, is the kernel.  And the set $\{R_s\}_{s \in S}$  is bounded.

By rotating \eqref{es1}, we get the exact triangle
\[
 H^0(I_s)[-1] \to H^{-1}(I_s)[1] \to I_s \to H^0(I_s),
\]
from which we see  $I_s$ is the cone of a morphism
\begin{align*}
\al &\in \Hom (H^0(I_s)[-1], H^{-1}(I_s)[1]) \\
&\cong \Hom (H^0(I_s),H^{-1}(I_s)[2]) \\
&=\Ext^2 (Q_s,H^{-1}(I_s)) \text{ if we write $Q_s := H^0(I_s)$}.
\end{align*}

Applying $\Hom (Q_s,-)$ to the short exact sequence \eqref{es3}, we get
 the exact sequence of cohomology
\[
\Ext^1 (Q_s,E^\vee) \to \Ext^1 (Q_s,R) \to \Ext^2 (Q_s,H^{-1}(I_s)) \to \Ext^2 (Q_s, E^\vee)
\]
where $\Ext^i (Q_s,E^\vee)\cong \Ext^i (Q_s\otimes E, \OO_X) =0$ for $i=1,2$ since $Q_s \otimes E$ is of codimension 3.  So we get an isomorphism $\Ext^1 (Q_s,R_s) \cong \Ext^2 (Q_s,H^{-1}(I_s))$.  This means that $\al$ is represented by the Yoneda product of the short exact sequence \eqref{es3} and an extension
\begin{equation}\label{es4}
0 \to R_s \to F_s \to Q_s \to 0.
\end{equation}

  Overall, the two-term complex $\{E^\vee \to F_s\}$, with $F$ at degree 0, obtained from composing the maps in \eqref{es3} and \eqref{es4}, is quasi-isomorphic to the cone of $\al$.  Since $I_s$ is  the cone of $\al$, we get that
\[
I_s \cong \{E^\vee \to F\}.
\]

Since the set of all $H^{-1}(I_s)$ is bounded, there is only a finite number of possibilities for the Hilbert polynomial of $H^{-1}(I_s)$.  And so from \eqref{es1}, the length of $Q=H^0(I_s)$ is also bounded.  Since 0-dimensional schemes on $X$ are parametrised by Hilbert schemes, the set $\{Q_s\}_{s \in S}$ is also bounded.  Hence from \eqref{es4}, the set $\{F_s\}_{s \in S}$ is bounded.  Hence the moduli space of PT-semistable objects with $ch=(-r,-d,\beta,n), r >0$, is bounded, and the statement of the proposition follows from \cite[Lemma 1.7.6]{HL}
\end{proof}

\section{Completeness of the Heart}

The goal of this section is to prove that, given an object $E_K \in \Ap_K$, we can extend it to an $R$-flat family of objects in $\Ap$.  This is the valuative criterion for completeness for the heart.

\subsection{Extending Semistable Objects}

Since every PT-semistable object of nonzero rank in $\Ap_K$ lies in the category $\langle \Coh_{\leq 0}(X_K), \Coh_{\geq 3}(X_K)[1]\rangle$ (Lemma \ref{lemma-PTsscharacterise}), we start with a weaker version of  completeness of the heart.

\begin{theorem}\label{theorem-completenessofheart-1}
Let $X$ be a smooth projective three-fold over $k$.  Given any object
 \[
 E_K \in \langle \Coh_{\leq d}(X_K), \Coh_{\geq 3}(X_K)[1]\rangle \subset D^b(X_K) \textnormal{ (where $0\leq d < 3$)},
 \]
 there exists a 3-term complex $\widetilde{E}^\bullet$ of $R$-flat coherent sheaves with $R$-flat cohomology on $X_R$ such that:
\begin{itemize}
 \item the generic fibre $j^\ast (\widetilde{E}^\bullet) \cong E_K$ in $D^b(X_K)$;
 \item the central fibre $L\iota^\ast (\widetilde{E}^\bullet) \in \langle \Coh_{\leq d}(X_k), \Coh_{\geq 3}(X_k)[1]\rangle \subset D^b(X_k)$.
\end{itemize}
\end{theorem}

\paragraph{Remark.}  In fact, this theorem says a little more than we really need in the rest of this paper.  After presenting a proof to this theorem, we will state and prove a bare-bone version of it.

\smallskip

Here is a construction that will be useful in the proof of Theorem \ref{theorem-completenessofheart-1}: Suppose $F, G$ are $R$-flat coherent sheaves on $X_R$, and $\phi : F \to G$ is a sheaf morphism such that $j^\ast (\phi) : j^\ast F \twoheadrightarrow j^\ast G$ is a surjection in $\Coh (X_K)$.  Consider the  exact sequence in $\Coh (X_R)$
\[
F \overset{\phi}{\to} G \to \cokernel (\phi) \to 0.
\]
Since $j^\ast \phi$ is surjective, we know $j^\ast (\cokernel (\phi))=0$, implying that $\cokernel \phi$ is supported on $X_m$ for some $m>0$.    Hence for large enough $s$, $\pi^m G$ is taken to zero by the quotient map $G \twoheadrightarrow \cokernel (\phi)$, i.e.\ $\pi^m G \subseteq \image \phi$.

\noindent\textit{Construction A.} Given the setup above, let $I$ be the pullback of the surjection $F \twoheadrightarrow \image \phi$ along the injection $\pi^m G \hookrightarrow \image \phi$, so that we have the pullback square of coherent sheaves on $X_R$

\begin{equation*}
\xymatrix{
  F \arsurj[r]^\phi & \image \phi  \\
  I \arinj[u] \arsurj[r] & \pi^m G \arinj[u]
}.
\end{equation*}
Note that $I$ is again an $R$-flat coherent sheaf.  On the other hand, for any $R$-flat coherent sheaf $A$ on $X_R$, the map $A \overset{\pi^m}{\to} \pi^m A$ of multiplication by $\pi^m$ is an isomorphism; denote the inverse of this map by multiplication by $1/\pi^m$.  Using such an isomorphism, we can construct a surjection $\phi'$ defined as the composition
\[
 \phi' : I \twoheadrightarrow \pi^m G \overset{\thicksim}{\to} G.
\]
Since $\pi^m F \subseteq I \subseteq F$ (to see the first inclusion, note that $\phi : F \to G$ takes $\pi^m F$ into $\pi^m G$, and so $\pi^m F \subseteq I$), we have $j^\ast I = j^\ast F$.  From now on, we will say `Construction A' to mean replacing a morphism of $R$-flat coherent sheaves $\phi : F \to G$ on $X_R$ by a surjection $\phi' : I \twoheadrightarrow G$ such that $j^\ast (\phi')$ is the composition of $j^\ast (\phi)$ followed by multiplication by $1/\pi^m$, for a suitable $m$.


\begin{proof}[Proof of Theorem \ref{theorem-completenessofheart-1}]
Suppose $E_K$ is represented by a two-term complex
\[
  E_K = [A_K \overset{s_K}{\to} B_K]
\]
where $A_K, B_K$ are coherent sheaves on $X_K$.  We can decompose this complex into two short exact sequences (which are the bottom row and left column of the following diagram)
\begin{equation*}
\xymatrix{
  \kernel (s_K) \arinj[d] & & \\
  A_K \ar[r]^{s_K} \arsurj[d]_{q^{-1}} & B_K \areq[d] & \\
  \image (s_K) \arinj[r] & B_K \arsurj[r]^<(0.3){q^0} & \cokernel (s_K)
}
\end{equation*}
The spirit of the proof is to extend the two short exact sequences to short exact sequences of $R$-flat coherent sheaves on $X_R$.

On $X_K$, we have the ample line bundle $\OO_X (1) \otimes_k K$.  Take any surjection $\OO_X(-m)\otimes_k K \twoheadrightarrow B_K$.  Then, by properness of the quot scheme, we can extend $B_K$ to an $R$-flat coherent sheaf $B_R$ on $X_R$.  Similarly, we can extend $\cokernel (s_K)$ to an $R$-flat sheaf $\widetilde{\cokernel (s_K)}$ on $X_R$, and subsequently extend $q^0$ to a morphism of $\OO_{X_R}$-modules $\bar{q^0} : B_R \to \wt{\cokernel (s_K)}$ such that $j^\ast (\bar{q^0}) = q^0$, and $L\iota^\ast \bar{q^0}$ is nonzero.  Let $\wt{q^0}$ be the map $B_R \to \image (\bar{q^0})$ obtained by restricting the codomain of $\bar{q^0}$.  Then $\image (\bar{q^0})$ is still $R$-flat, and we still have $j^\ast (\wt{q^0}) = q^0$ and $L\iota^\ast (\wt{q^0}) \neq 0$.

Now, the coherent sheaf $\kernel (\wt{q^0})$ is $R$-flat and restricts to $\kernel (q^0) = \image (s_K)$ on $X_K$.

Choose any $R$-flat extension $A_R$ of $A_K$ on $X_R$, extend $q^{-1}$ to a morphism $A_R \to \kernel (\wt{q^0})$ on $X_R$, and apply Construction A to obtain a surjection $\wt{q^{-1}} : I \twoheadrightarrow \kernel (\wt{q^0})$, where $I$ is an $R$-flat subsheaf of $A_R$.
Here, $j^\ast (\wt{q^{-1}}) = \frac{1}{\pi^m} \circ q^{-1}$ for some $m$, and so $\wt{q^{-1}}$ does not quite restrict to $q^{-1}$ on $X_K$.  To get around this, simply replace the differential $s_K$ by $\frac{1}{\pi^m} \circ s_K$ (this replaces $E_K$ by a quasi-isomorphic complex).  Then, we truly have $j^\ast (\wt{q^{-1}} )= q^{-1}$.

 Now, $\kernel (\wt{q^{-1}})$ restricts to $\kernel (s_K)$ on $X_K$, but $\iota^\ast (\kernel (\wt{q^{-1}}))$ may not be torsion-free on $X_k$.  So we cannot just take the 2-term complex $[I \to B_R]$, with the differential being $\wt{q^{-1}}$ followed by the inclusion $\kernel (\wt{q^0}) \hookrightarrow B_R$, to be our extension of $E_K$ on $X_R$, for then the derived restriction to $X_k$ may not be in $\langle \Coh_{\leq d}(X_k), \Coh_{\geq 3}(X_k)[1]\rangle$.

To rectify this, use Langton's construction \cite[Proposition 6]{Langton} to extend $\kernel (s_K)$ to a torsion-free sheaf $\wt{\kernel (s_K)}$ on $X_R$ such that its pullback to $X_k$ is still a torsion-free sheaf.  Then, choose any surjection $f : F \twoheadrightarrow \wt{\kernel (s_K)}$ where $F$ is a locally free sheaf on $X_R$.  Also, extend the identity map of $\kernel (s_K)$ to a morphism $\wt{1} : \kernel (\wt{q^{-1}}) \to \wt{\kernel (s_K)}$ on $X_R$.  Then, define $M$ to be the kernel of $F \oplus \kernel (\wt{q^{-1}}) \overset{(f,\wt{1})}{\twoheadrightarrow} \wt{\kernel (s_K)}$.  Then $M$ is an $R$-flat coherent sheaf on $X_R$.

So far, we have constructed the following commutative diagram in which each three-term straight-line sequence is a short exact sequence, and all the terms are $R$-flat sheaves:
\begin{equation*}
\xymatrix{
 M\arinj[r]^<(0.2){i_1} &  F \oplus \kernel (\wt{q^{-1}}) \arinj[d]^{1_F\oplus i_2} \arsurj[r]^{(f,\wt{1})}  &  \wt{\kernel (s_K)} & \\
 & F \oplus I \ar[r] \arsurj[d]_{(0,\wt{q^{-1}})} & B_R \areq[d] & \\
&  \kernel (\wt{q^0})  \arinj[r]^{i_3} & B_R \arsurj[r]^<(0.3){\wt{q^0}} & \image (\bar{q^0})
}.
\end{equation*}
Here, $i_1, i_2$ and $i_3$ are the canonical inclusions.

If we define $\wt{E}^{-2} := M$, $\wt{E}^{-1} := F \oplus I$ and $\wt{E}^0 := B_R$, and take $\wt{s} = i_3 \circ (0,\wt{q^{-1}})$, then
\[
\wt{E}^\bullet = [\wt{E}^{-2} \hookrightarrow \wt{E}^{-1} \overset{\wt{s}}{\to} \wt{E}^0]
 \]
 is a 3-term complex of $R$-flat coherent sheaves.  Therefore, $L\iota^\ast \wt{E}^\bullet = \iota^\ast \wt{E}^\bullet$.  Moreover, from our construction, all the cohomology sheaves of $\wt{E}^\bullet$ are flat over $R$; as a consequence, the cohomology sheaves of $\iota^\ast \wt{E}^\bullet$ are simply the pullback of the  cohomology sheaves of $\wt{E}^\bullet$.  And so $L\iota^\ast \wt{E}^\bullet$  is an object in the heart $\langle \Coh_{\leq d}(X_k), \Coh_{\geq 3}(X_k)[1]\rangle$.

It remains to show that $j^\ast \wt{E}^\bullet$ is quasi-isomorphic to $E_K$.  This is not hard to see.  Note that $j^\ast \wt{E}^\bullet = [j^\ast M \hookrightarrow j^\ast F \oplus j^\ast I \overset{j^\ast \wt{s}}{\to} j^\ast B_R]$ where $j^\ast I = A_K$, and $j^\ast \wt{s} = (0,s_K)$.  Define a chain map $c^\bullet : E_K \to j^\ast \wt{E}^\bullet$ by the commutative diagram
\begin{equation*}
\xymatrix{
 j^\ast M \arinj[r] &  j^\ast F \oplus j^\ast I \ar[r]^{j^\ast \wt{s}} & j^\ast B_R \\
0 \ar[u] \ar[r] &   A_K \ar[u]^{c^{-1}} \ar[r]^{s_K} & B_K \ar[u]^{c^0}
}
\end{equation*}
where $c^{-1}$ is the canonical injection into the second factor, and $c^0$ is the identity map. That $H^0(c^\bullet)$ is an isomorphism is clear.  The map
\[
  H^{-1}(c^\bullet) : \kernel (s_K) \to \frac{\kernel (j^\ast \wt{s})}{\image j^\ast (i_1)} \cong j^\ast \wt{\kernel (s_K)}
\]
is an injection between two isomorphic coherent sheaves on $X_K$, a projective scheme over $K$, and so is an isomorphism.  Therefore, $c^\bullet$ is a quasi-isomorphism, and so $j^\ast \wt{E}^\bullet \cong E_K$.  This completes the proof of the theorem.
\end{proof}

As mentioned, we really only need the following bare-bone version of Theorem \ref{theorem-completenessofheart-1}:

\begin{pro}\label{pro-completenessofheart-1-v2}
Let $X$ be a smooth projective three-fold over $k$.  Given any object
 \[
 E_K \in \langle \Coh_{\leq d}(X_K), \Coh_{\geq 3}(X_K)[1]\rangle \subset D^b(X_K) \textnormal{ (where $0\leq d < 3$)},
 \]
 there exists an object  $\widetilde{E} \in D^b(X_R)$ such that:
\begin{itemize}
 \item the generic fibre $j^\ast (\widetilde{E}) \cong E_K$ in $D^b(X_K)$;
 \item the central fibre $L\iota^\ast (\widetilde{E}^\bullet) \in \langle \Coh_{\leq d}(X_k), \Coh_{\geq 3}(X_k)[1]\rangle \subset D^b(X_k)$.
\end{itemize}
\end{pro}

The proof of this proposition relies on one technical lemma:

\begin{lemma}\label{extension-completeness}
Suppose $S_1, S_3$ are two stacks of objects in $D^b(X)$ that satisfy the valuative criterion for completeness.  Suppose $S_2$ is a stack whose points are objects $s_2$  given by extensions of the form
\[
  s_1 \to s_2 \to s_3 \to s_1[1] \text{\quad in $D^b(X)$}
\]
where $s_1, s_3$ are points in $S_1, S_3$, respectively.  Then $S_2$ also satisfies the valuative criterion for completeness.
\end{lemma}

\begin{proof}
Suppose we have an extension
\[
  s_1 \to s_2 \to s_3 \overset{\al}{\to} s_1[1]
\]
where $s_i \in S_i (\Spec K)$ for all $i$.  By hypothesis, for $i=1$ and $3$, there exist $R$-flat $\tilde{s_i} \in S_i (\Spec R)$ that restrict to $s_i$ over $\Spec K$.

Note that, for objects $F, G \in D^b(X_R)$, the $R$-module $\Hom_{D^b(X_R)} (F,G)$ is finitely generated and
\begin{equation}\label{HomX_RX_K}
   \Hom_{D^b(X_R)}(F,G) \otimes_R K \cong \Hom_{D^b(X_K)} (F \otimes_R K, G\otimes_R K)
\end{equation}
as in \cite[Lemma 3.18]{TodaLSOp}.  Therefore, there exists an integer $m \geq 0$ such that $\pi^m \al$ extends to a morphism $\tilde{\al} : \tilde{s}_3 \to \tilde{s}_1[1]$ over $\Spec R$.  Even though $L\iota^\ast \tilde{\al}$ may be zero, in which case $\cone (L\iota^\ast \tilde{\al})[-1] \cong L\iota^\ast \tilde{s}_1 \oplus L\iota^\ast \tilde{s}_3$ is a direct sum, this is of no concern to us: if we define $\tilde{s}_2$ to be $\cone (\al)[-1]$, then $L\iota^\ast \tilde{s}_2 \in S_2$, and $j^\ast (\tilde{s}_2) \cong j^\ast \cone (\al)[-1] \cong s_2$.  This shows that $S_2$ also satisfies the valuative criterion for completeness.
\end{proof}

\begin{proof}[Proof of Proposition \ref{pro-completenessofheart-1-v2}]
The stacks $\Coh_{\leq d}(X)$ and $\Coh_{\geq 3}(X)$ on a three-fold $X$ both satisfy the valuative criterion for completeness, so the proposition follows from Lemma \ref{extension-completeness}.
\end{proof}

We easily obtain:

\begin{coro}\label{extension-completeness-coro1}
Let $(\mathcal{T},\mathcal{F})$ be a stack of torsion theories in the sense of  \cite[Appendix]{BSMSKTS}, where $\mathcal{T}$ and $\mathcal{F}$ are substacks of a stack of abelian groups $\mathcal{A}$ both of which satisfy the valuative criterion for completeness.  Then the stack of tilted objects $\langle \mathcal{T}, \mathcal{F}[1]\rangle$ also satisfies the valuative criterion for completeness.
\end{coro}

\begin{coro}\label{XKssXkinheart}
Given a PT-semistable object $E_K \in \Ap_K$ of nonzero rank, there exists a flat family $E$ of objects in $\Ap$ over $\Spec R$ such that $j^\ast E \cong E_K$.
\end{coro}

\subsection{Completeness of the Heart $\Ap$}

As mentioned at the start of this section, we can prove more: that any object in $\Ap_K$, not just PT-semistable objects, can be extended to a flat family over $\Spec R$.

\begin{pro}\label{pro-completenessofheart-1}
Let $E_K \in \Ap_K$ be any PT-semistable object.  Then there exists an object $I \in D^b(X_R)$ such that $j^\ast I \cong E_K$ and $L\iota^\ast I \in\Ap$.
\end{pro}

\begin{proof}
Suppose $\rank (E_K)= 0$.  Then $E_K$ is just a sheaf by Lemma \ref{lemma-rankzeroPTssobj}, and the result follows from the usual valuative criterion for completeness for semistable sheaves.  Otherwise, the result follows from Corollary \ref{XKssXkinheart}.
\end{proof}

We are now ready to prove the completeness of the heart $\Ap = \langle \Coh_{\leq 1}, \Coh_{\geq 2}[1]\rangle$, which is more general than Theorem \ref{theorem-completenessofheart-1}.

\begin{theorem}[Completeness of the heart]\label{theorem-completenessofheart-3}
Let $X$ be a smooth projective three-fold over $k$.  Suppose $E_K \in \Ap_K$.  Then there exists some $I \in D^b(X_R)$ such that $j^\ast I \cong E_K$ in $D^b(X_K)$ and $L\iota^\ast I \in \Ap (X_k)$.
\end{theorem}

\begin{proof}
All we have to do is to  consider the HN filtration of $E_K$, extend the semistable quotients one by one, and piece them back together.  More explicitly, suppose $E_K$ has the following HN filtration with respect to PT-stability on $X_K$ (HN filtrations for polynomial stability exist by \cite[Theorem 3.2.2]{BayerPBSC}): $0=G^0_K \subset G^1_K \subset \cdots \subset G^n_K = E_K$.

Since each $G^i_K$ is PT-semistable, the result follows from Proposition \ref{pro-completenessofheart-1} and Lemma \ref{extension-completeness}.
\end{proof}

\section{Computing Cohomology with Respect to the Heart $\Ap$}

In performing semistable reduction for a flat family of complexes in the derived category (as we will be doing in the sequel to this paper),  we will often pull back or push forward the complexes at hand, and then compute their cohomology with respect to the t-structure given by $\Ap_m$ for some $m$.  Here, we establish  technical tools tailored for these situations.

\subsection{$t$-Structures on the Unbounded Derived Category}

As an example, let $G \in \Ap_m$, where $m \geq 1$.  Then $\iotam_\ast G \in D(X_R)$.  \textit{A priori}, we do not know what $L\iota_m^\ast \iotam_\ast G$ looks like: we know that the derived pullback $L\iota_m^\ast(-)$ is the same as the derived tensor $-\Lo_{\OO_{X_R}} \OO_{X_m}$, but we do not know what a locally free resolution of $\iotam_\ast G$ on $X_R$ looks like in general.  To get around this problem, we note that $\iotam_\ast$ is an exact functor, and so preserves cohomology.  That is, the pushforwards of the cohomology sheaves of $L\iota_m^\ast \iotam_\ast G$ will be the cohomology sheaves of $\iotam_\ast L\iota_m^\ast \iotam_\ast G$, which is in $D(X_R)$.  Then we can work out what $\iotam_\ast L\iota_m^\ast (\iotam_\ast G) \cong \iotam_\ast G \Lo_{\OO_{X_R}} \OO_{X_m}$ looks like by replacing $\OO_{X_m}$ with the resolution $[\OO_{X_R} \overset{\pi^m}{\to} \OO_{X_R}]$.  Then
\begin{align*}
\iotam_\ast L\iota_m^\ast (\iotam_\ast G) &\cong \iotam_\ast G \Lo_{\OO_{X_R}} \OO_{X_m} \\
&\cong \iotam_\ast G \Lo_{\OO_{X_R}} [\OO_{X_R} \overset{\pi^m}{\to} \OO_{X_R}] \\
&\cong \iotam_\ast G [1] \oplus \iotam_\ast G \text{\quad in $D(X_R)$}.
\end{align*}

As another example, let $1 \leq m' < m$, and let $G \in \Ap_{m'}$.  Then we can similarly try to understand the cohomology of $L\iota_{m,m'}^\ast {\iota_{m,m'}}_\ast G$ by understanding the cohomology of ${\iota_{m,m'}}_\ast L\iota_{m,m'}^\ast {\iota_{m,m'}}_\ast G$.  Note that $\OO_{X_{m'}}$ has an infinite locally free resolution on $\OO_{X_m}$,
\[
\OO_{X_{m'}} \cong [ \cdots \to \OO_{X_m} \overset{\pi^{m'}}{\longrightarrow} \OO_{X_m} \overset{\pi^{m-m'}}{\longrightarrow} \OO_{X_m} \overset{\pi^{m'}}{\longrightarrow} \OO_{X_m}] \text{\quad in $D(X_m)$}.
\]
If $m'=1$, for instance, then multiplication by $\pi^{m-m'}=\pi^{m-1}$ or $\pi$ would induce the zero map from $G$ to itself, in which case ${\iota_{m,1}}_\ast L\iota_{m,1}^\ast {\iota_{m,1}}_\ast G \cong \oplus_{i \geq 0} {\iota_{m,1}}_\ast G[i]$.

From the last example, we see that even when we start with an object in the bounded derived category, we can still easily end up with an object with unbounded cohomology.  Therefore, it would  be useful to know that the definition of our t-structures (corresponding to the hearts $\Ap_m$) has nothing to do with whether the ambient derived category is bounded or unbounded.  In other words, we want to make sure that we can talk about cohomology of unbounded complexes with respect to the t-structures $\Ap_m$.  This is the content of the next proposition.

Let $(\mathcal{T},\mathcal{F})$ be a torsion pair in an abelian category $\mathcal{A}$.  From \cite[Proposition 2.1]{TACQA}, we know that the pair
\begin{align*}
  D^{\leq 0,b} &= \{ E \in D^b(\mathcal{A}) : H^0(E) \in \mathcal{T}, H^i(E)=0 \, \forall i>0\} \\
  D^{\geq 0,b} &= \{E \in D^b(\mathcal{A}) : H^{-1}(E) \in \mathcal{F}, H^i(E)=0\, \forall i<-1\}
\end{align*}
define a t-structure on the bounded derived category $D^b(\mathcal{A})$.  We claim that

\begin{pro}\label{tstructureDX}
The two subcategories
\begin{align*}
  D^{\leq 0} &= \{ E \in D(\mathcal{A}) : H^0(E) \in \mathcal{T}, H^i(E)=0 \, \forall i>0\} \\
  D^{\geq 0} &= \{E \in D(\mathcal{A}) : H^{-1}(E) \in \mathcal{F}, H^i(E)=0\, \forall i<-1\}
\end{align*}
of $D(\mathcal{A})$ define a t-structure on the unbounded derived category $D(\mathcal{A})$.
\end{pro}

\begin{proof}
 Let us just prove axiom (3) in the definition of a t-structure (see section \ref{section-tstructures} for the axioms).  The proofs of the other axioms are easier to show.  Furthermore, let us just prove the third axiom for the bounded above derived category $D^-(\mathcal{A})$; it is straightforward to extend the proof to the case of the unbounded derived category $D(\mathcal{A})$.  The philosophy of the proof is to truncate the complexes first, then apply the results for bounded complexes, and finally put the complexes back together using the octahedral axiom.  In this proof, we will use $\tau^{\leq i}, \tau^{\geq i}$ to denote truncation functors with respect to the t-structure $(D^{\leq 0,b}, D^{\geq 0,b})$, and use $\tau^{\leq i}_{\mathcal{A}}, \tau^{\geq i}_{\mathcal{A}}$ to denote the truncations with respect to the standard t-structure on $D^-(\mathcal{A})$.

  Axiom (3) of a t-structure says that, given any object $E \in D^-(\mathcal{A})$, there should exist an exact triangle $A \to E \to B \to A[1]$ where $A \in D^{\leq 0}, B \in D^{\geq 1}$.  Pick any $E \in D^-(\mathcal{A})$.  Pick an integer $k \ll 0$.  Then we have the exact triangle $\tau_\mathcal{A}^{\leq k-1} E \to E \stackrel{\al}{\to}  \tau_{\mathcal{A}}^{\geq k} E  \to \tau_\mathcal{A}^{\leq k-1} E [1]$.  Let $Y := \tau_{\mathcal{A}}^{\geq k} E$, which is an object in $D^b(\mathcal{A})$.  By what we know about the bounded case, we have an exact triangle $\tau^{\leq -1} Y \to Y \stackrel{\beta}{\to}  \tau^{\geq 0} Y  \to \tau^{\leq -1} Y [1]$.

Now we can stack the first triangle on top of the second one as follows:
\begin{equation*}
\xymatrix{
\ar@{.}[rr] & & \tau_{\mathcal{A}}^{\leq k-1} E \ar[rr] \ar@{.}[dl] & & E \ar[dl]_\al \ar@/^2pc/@{.>}[ddll]^{\beta\al} \\
 & \tau^{\leq -1}Y \ar[rr] \ar@{.}[ul]& & \tau_{\mathcal{A}}^{\geq k} E= Y  \ar[ul]^{[1]} \ar[dl]_\beta & \\
& & \tau^{\geq 0} Y \ar[ul]^{[1]} & &
}.
\end{equation*}
In this diagram, the triangle completing the morphism $E \overset{\beta\al}{\longrightarrow} \tau^{\geq 0}Y$ is exactly the triangle we are looking for.  We can check this using the octahedral axiom: start with the commutative triangle
\begin{equation*}
\xymatrix{
 & Y \ar[dr]^\beta & \\
  E \ar[rr]^{\beta \al} \ar[ur]^\al & & \tau^{\geq 0}Y
}.
\end{equation*}
The octahedral axiom then gives the following diagram
\begin{equation*}
\def\objectstyle{\scriptstyle}
\def\labelstyle{\scriptstyle}
\xymatrix@R=2pc{
 & & & \tau_{\mathcal{A}}^{\leq k-1}E[1] \ar[ddd] \\
 & & & \\
& Y \ar[rruu] \ar[dr]^\beta & & \\
E \ar[ur]^\al \ar[rr]^{\beta \al} & & \tau^{\geq 0}Y \ar[r] \ar[dr] & Z[1] \ar[d] \\
& & & \tau^{\leq -1}Y[1]
}
\end{equation*}
where each straight line is an exact triangle, and where $Z$ is defined as $\cone (\beta \al)[1]$.

In the  exact triangle formed by the vertical line, we note that the standard cohomology of $\tau_{\mathcal{A}}^{\leq k-1}E[1]$ are all zero near degree 0 (since $k \ll 0$), and so the standard cohomology of $Z$ and $\tau^{\leq -1}Y$ all agree near degree 0.  So $Z \in D^{\leq -1}$.  Thus $Z \to E \to \tau^{\geq 0}Y \to Z[1]$ is the desired exact triangle.
\end{proof}

\subsection{Properties of Pullbacks and Pushforwards}\label{section-ppp}

\begin{lemma}\label{cohomology-lemma0}
For integers $1\leq m' < m$, the derived pullback $L\iota_{m,m'}^\ast : D^-(X_m) \to D^-(X_{m'})$ between the bounded above derived categories is right t-exact with respect to the t-structures corresponding to $\Ap_m$ and $\Ap_{m'}$, as is the functor $L\iota_{m'}^\ast \iotam_\ast : D^-(X_m) \to D^-(X_{m'})$.
\end{lemma}

Recall that right t-exactness in this case means $L\iota_{m,m'}^\ast$ takes $D^{\leq 0}_{\Ap_m}$ into $D^{\leq 0}_{\Ap_{m'}}$.

\begin{proof}
By Proposition \ref{tstructureDX}, we know we can describe $D^{\leq 0}_{\Ap_m}(X_m)$ as
\[
  D^{\leq 0}_{\Ap_m}(X_m) = \{ E \in D(X_m) : H^i (E)=0 \text{ for $i >0$}, H^0(E) \in \Coh_{\leq 1}(X_m)\}.
\]
So by Lemmas \ref{flatness-lemma1} and \ref{flatness-lemma2}, proved below, we get that $E \in D^{\leq 0}_{\Ap_m}(X_m)$ implies $L\iota_{m,m'}^\ast (E) \in D^{\leq 0}_{\Ap_{m'}}(X_{m'})$.  Then the right t-exactness of $L\iota_{m,m'}^\ast : D^-(X_m) \to D^-(X_{m'})$ follows.

The second part of the assertion follows from $L\iota_{m'}^\ast \iotam_\ast \cong L\iota_{m,m'}^\ast L\iota_m^\ast \iotam_\ast$ and the obvious right t-exactness of $L\iota_m^\ast \iotam_\ast$.
\end{proof}

\begin{coro}\label{cohomology-lemma1}
The composite functor $\mathcal{H}^0_{\Ap_{m'}} \circ L\iota_{m,m'}^\ast : D^-(X_m) \to D^-(X_{m'})$ is right t-exact with respect to $\Ap_m$ and $\Ap_{m'}$.
\end{coro}

An advantage of $\mathcal{H}^0_{\Ap_{m'}} \circ L\iota_{m,m'}^\ast$ as opposed to $L\iota_{m,m'}^\ast$ is that it takes objects in the heart $\Ap_m$ to the heart $\Ap_{m'}$.  We also have an analogue of $\iota_{m,m''}^\ast \cong \iota_{m',m''}^\ast \circ \iota_{m,m'}^\ast$ (for $1\leq m''<m'<m$):

\begin{coro}\label{cohomology-lemma2}
For $1 \leq m'' <m'<m$, we have the following isomorphism of functors from $D^{\leq 0}_{\Ap_m}(X_m)$ to $D^{\leq 0}_{\Ap_{m''}}(X_{m''})$:
\[
\mathcal{H}^0_{\Ap_{m''}} \circ L\iota_{m,m''}^\ast  \cong (\mathcal{H}^0_{\Ap_{m''}} \circ L\iota_{m',m''}^\ast) \circ (\mathcal{H}^0_{\Ap_{m'}} \circ L\iota_{m,m'}^\ast) .
\]
\end{coro}

\begin{proof}
First we make an observation - take any morphism $E \overset{f}{\to} F$ in $D^{\leq 0}_{\Ap_m}(X_m)$.  Then we get a morphism between exact triangles (where the truncation functors are with respect to $\Ap_m$)

\begin{equation*}
\xymatrix{
\tau^{\leq -1} E \ar[r] \ar[d] & E \ar[d]^f \ar[r] & \tau^{\geq 0}E \ar[d]^{\tau^{\geq 0}(f)} \ar[r]  & \tau^{\leq -1} E[1] \ar[d] \\
\tau^{\leq -1} F \ar[r]        & F          \ar[r] & \tau^{\geq 0}F \ar[r] & \tau^{\leq -1}F[1]
}.
\end{equation*}
That the connecting morphism is functorial is part of \cite[Proposition 10.1.4]{SM}.

If we apply $L\iota_{m,m'}^\ast$ to the above diagram, noting that it is right t-exact, and then apply $\tau^{\geq 0}_{\Ap_{m'}}$, then the left-most and right-most columns of the diagram vanish, and the horizontal maps that remain are isomorphisms, and so we have an isomorphism of functors
\[
  \tau^{\geq 0}_{\Ap_{m'}} L\iota_{m,m'}^\ast \cong \tau^{\geq 0}_{\Ap_{m'}} L\iota_{m,m'}^\ast \tau^{\geq 0}_{\Ap_m} : D^{\leq 0}_{\Ap_m}(X_m) \to D^{\leq 0}_{\Ap_{m'}}(X_{m'}).
\]

 Dropping the subscripts in the truncation functors, we have the following series of isomorphisms of functors $D^{\leq 0}_{\Ap_m}(X_m) \to D^{\leq 0}_{\Ap_{m''}}(X_{m''})$:
\begin{align*}
  \mathcal{H}^0 L\iota_{m,m''}^\ast &\cong \tau^{\leq 0}(\tau^{\geq 0} L\iota_{m,m''}^\ast ) \text{ by definition of $\mathcal{H}^0$}\\
&\cong \tau^{\leq 0} (\tau^{\geq 0} L\iota_{m',m''}^\ast ) L\iota_{m,m'}^\ast \\
&\cong \tau^{\leq 0} (\tau^{\geq 0} L\iota_{m',m''}^\ast \tau^{\geq 0})L\iota_{m,m'}^\ast \text{ by the observation above} \\
&= \tau^{\leq 0} \tau^{\geq 0} L\iota_{m',m''}^\ast (\tau^{\geq 0} L\iota_{m,m'}^\ast) \\
&\cong \tau^{\leq 0}\tau^{\geq 0} L\iota_{m',m''}^\ast (\tau^{\leq 0} \tau^{\geq 0} L\iota_{m,m'}^\ast) \text{ by Lemma \ref{cohomology-lemma0}} \\
&= \mathcal{H}^0 L\iota_{m',m''}^\ast \mathcal{H}^0 L\iota_{m,m'}^\ast \text{ as wanted.}
\end{align*}
\end{proof}

\begin{lemma}\label{cohomology-lemma3}
For any $E \in \Ap_m$, if $\HH^0 L\iota_{m,1}^\ast E =0$ then $E=0$.
\end{lemma}
\begin{proof}
Suppose $\HH^0 L\iota_{m,1}^\ast E=0$.  Since $\HH^0 L\iota_{m,1}^\ast$ is right exact, we have $L\iota_{m,1}^\ast E \in D^{\leq -1}_{\Ap}$.  So the cohomology $H^0(L\iota_{m,1}^\ast E)$ is zero.  However, $H^0(L\iota_{m,1}^\ast E) \cong \iota_{m,1}^\ast H^0(E)$, so we in fact have $H^0(E)=0$ by \cite[Lemma 2.1.3]{Lieblich}.  We also get $H^{-1}(L\iota_{m,1}^\ast E) \in\Coh_{\leq 1}(X_k)$, so by Lemma \ref{flatness-lemma1} below we know $H^{-1}(E) \in \Coh_{\leq 1} (X_m)$, i.e.\ $E \in D^{\leq -1}_{\Ap_m}$, forcing $E=0$ in $\Ap_m$.
\end{proof}

\begin{lemma}\label{cohomology-lemma5}
(a) For any $m \geq 1$, we have an isomorphism of functors from $\Ap_m$ to $\Ap_m$:
$$\HH^0_{\Ap_m} L\iota_m^\ast \iotam_\ast \overset{\thicksim}{\to} \text{id}_{\Ap_m}.$$

(b) For any $1\leq m' < m$, we have an isomorphism of functors from $\Ap_{m'}$ to $\Ap_{m'}$:
$$ \HH^0_{\Ap_{m'}} L\iota_{m,m'}^\ast {\iota_{m,m'}}_\ast \overset{\thicksim}{\to} \text{id}_{\Ap_{m'}}.$$
\end{lemma}
\begin{proof}
Let us  prove part (a) - the proof of part (b) is analogous.  To start with, note that we have the adjoint pair $L\iota_m^\ast \dashv \iotam_\ast$.  Therefore, we have a morphism of functors $L\iota_m^\ast \iotam_\ast \to \text{id}_{D^-(X_m)}$.  In particular, this means that for any morphism $Y \overset{f}{\to} Z$ in $\Ap_m \subset D^-(X_m)$ we have a commutative diagram in $D^-(X_m)$

\begin{equation*}
\xymatrix{
  L\iota_m^\ast \iotam_\ast Y \ar[rr]^{L\iota_m^\ast \iotam_\ast f} \ar[d]^{\theta_Y} & & L\iota_m^\ast \iotam_\ast Z \ar[d]^{\theta_Z} \\
  Y \ar[rr]^f && Z
}.
\end{equation*}
Applying the truncation functor $\tau^{\geq 0}_{\Ap_m}$ to the whole diagram, we get a commutative diagram in $\Ap_m$
\begin{equation*}
\xymatrix{
  \tau^{\geq 0} L\iota_m^\ast \iotam_\ast Y \ar[rr]^{\tau^{\geq 0} L\iota_m^\ast \iotam_\ast f} \ar[d]^{\tau^{\geq 0}\theta_Y} & & \tau^{\geq 0} L\iota_m^\ast \iotam_\ast Z \ar[d]^{\tau^{\geq 0}\theta_Z} \\
  Y \ar[rr]^{f} && Z
}.
\end{equation*}
By the exactness of $\iotam_\ast$ and right t-exactness of $L\iota_m^\ast \iotam_\ast$, we have $\tau^{\geq 0} L\iota_m^\ast \iotam_\ast Y = \HH^0 L\iota_m^\ast \iotam_\ast Y$, and similarly for $Z$.  If we can now show that $\tau^{\geq 0}\theta_Y$ and $\tau^{\geq 0}\theta_Z$ are isomorphisms,  we would be done.  Let us just check this for $\tau^{\geq 0}\theta_Y$.  Since quasi-isomorphisms are isomorphisms in the derived category, it would be enough to show that $H^i (\tau^{\geq 0} \theta_Y)$ is an isomorphism for all $i$.

Since the composition
\[
\xymatrix{
 \iotam_\ast Y \ar[r] & \iotam_\ast L\iota_m^\ast \iotam_\ast Y \ar[rr]^{\iotam_\ast \theta_Y}  &&\iotam_\ast Y
 }
 \]
(the first map being adjunction) is an isomorphism by \cite[(1.5.9)]{CS}, and $H^0(\iotam_\ast L\iota_m^\ast \iotam_\ast Y) \cong  H^0(\iotam_\ast Y)$, we have that $H^0 (\iotam_\ast \theta_Y)$ is a surjection between isomorphic sheaves on $X_R$.  By the exactness of $\iotam_\ast$, we see $H^0 (\theta_Y)$ must have been a surjection between isomorphic sheaves on $X_m$ to start with.  Comparing Hilbert polynomials, we conclude that $H^0(\theta_Y)=H^0(\tau^{\geq 0}_{\Ap_m} \theta_Y)$ is an isomorphism.

Lastly, we show that $H^{-1}(\tau^{\geq 0}\theta_Y)$ is also an isomorphism.  If we apply $H^{-1}$ to the sequence of maps $\iotam_\ast Y \to \iotam_\ast L\iota_m^\ast \iotam_\ast Y \to \iotam_\ast Y$, we obtain the sequence

\[
\xymatrix{
\iotam_\ast H^{-1}(Y)\ar[r] & \iotam_\ast H^0(Y) \oplus \iotam_\ast H^{-1}(Y) \ar[rr]^(.6){\iotam_\ast H^{-1}(\theta_Y)} && \iotam_\ast H^{-1}(Y)
}
\]
whose composition is an isomorphism.  Thus $\iotam_\ast H^{-1}(\theta_Y)$ is a surjection.  So $H^{-1}(\theta_Y)$ must have been a surjection to start with.

On the other hand, we have the commutative diagram (see the proof of \cite[Lemma 5(a)]{GM})
$$\xymatrix{
L\iota_m^\ast \iotam_\ast Y \ar[r]^(.4)\ep \ar[d]^{\theta_Y} & \tau^{\geq 0} L\iotam_\ast \iotam_\ast Y    \ar[d]^{\tau^{\geq 0}\theta_Y} \\
Y \ar[r]^(.4)1 & \tau^{\geq 0}Y = Y
}.$$
If we can show that $H^{-1}(\ep)$ is surjective, then $H^{-1}(\tau^{\geq 0}\theta_Y)$ would be a surjection between isomorphic sheaves on $X_m$, and hence an isomorphism, and we would be done.  However, if we look at the canonical exact triangle
$$ \tau^{\leq -1} L\iota_m^\ast \iotam_\ast Y \to L\iota_m^\ast \iotam_\ast Y \overset{\ep}{\to} \tau^{\geq 0} L\iota_m^\ast \iotam_\ast Y \to \tau^{\leq -1} L\iota_m^\ast \iotam_\ast Y [1]$$
and take its long exact sequence of cohomology with respect to the standard t-structure, we get
$$ \cdots \to H^0(Y) \to H^{-1}(L\iota^\ast_m \iotam_\ast Y) \overset{H^{-1}(\ep)}{\longrightarrow} H^{-1}(Y) \to 0 \to \cdots .$$
And so $H^{-1}(\ep)$ is a surjection, as we wanted.

For part (b), we start with the adjoint pair $L\iota_{m,m'}^\ast \dashv {\iota_{m,m'}}_\ast$.  In this case, we need Lemma \ref{cohomology-lemma0}.  The same proof still works.
\end{proof}

\noindent\textit{Remark.}  The proof above illustrates the philosophy, that if we find it difficult to compare cohomology of objects in $D(X_m)$, we can always push them forward to $D(X_R)$ and compare cohomology there, provided that there is compatibility of the t-structures involved.  (In this case, we are using the standard t-structures on $D(X_m)$ and $D(X_R)$.)

\begin{lemma}\label{cohomology-lemma4}
(a) Let $1 \leq m$.  Given any $A,B \in \Ap_m$, we have
\[
\Hom_{\Ap_m} (A,B) \overset{\thicksim}{\to} \Hom_{D(X_R)} (\iotam_\ast A, \iotam_\ast B)
\]
and this isomorphism is given by $f \mapsto \iotam_\ast f$. \newline
(b) Let $1 \leq m' < m$.  Given any $A,B \in\Ap_{m'}$, we have
\[
\Hom_{\Ap_{m'}} (A,B) \overset{\thicksim}{\to} \Hom_{\Ap_m} ({\iota_{m,m'}}_\ast A, {\iota_{m,m'}}_\ast B)
\]
and this isomorphism is given by $f \mapsto {\iota_{m,m'}}_\ast (f)$. \newline
\end{lemma}

\begin{proof}
Let us just prove part (b) - the proof of part (a) is analogous.  For $A,B \in \Ap_{m'}$, we have
\begin{align*}
  \Hom_{\Ap_{m'}} (A,B) &\cong \Hom_{D(X_{m'})} (A,B) \\
  &\cong \Hom_{D(X_{m'})} (\HH^0 L\iota_{m,m'}^\ast {\iota_{m,m'}}_\ast A,B) \text{ by Lemma \ref{cohomology-lemma5}(b)}\\
&\cong \Hom_{D(X_{m'})} (L\iota_{m,m'}^\ast {\iota_{m,m'}}_\ast A, B) \text{ by \cite[Proposition 10.1.4(i)]{SM}}\\
&\cong \Hom_{D(X_m)}({\iota_{m,m'}}_\ast A, {\iota_{m,m'}}_\ast B) \text{ by adjunction} \\
&= \Hom_{\Ap_m} ({\iota_{m,m'}}_\ast A, {\iota_{m,m'}}_\ast B) .
\end{align*}

In fact, given $f \in \Hom_{\Ap_{m'}}(A,B)$, the images of $f$ in the various sets above are related by the following diagram, where $\eta$ is the adjunction map
$$\xymatrix{
  \HH^0 L\iota_{m,m'}^\ast {\iota_{m,m'}}_\ast A \ar[rr]^(.6){\HH^0(\eta)} & & A \ar[r]^f & B \\
  L\iota^\ast_{m,m'} {\iota_{m,m'}}_\ast A \ar[urr]_\eta \ar[u] & & &
}.$$
\end{proof}


\begin{coro}\label{exacttriangleApmXm}
For any object $F \in \Ap_m$, we have the exact triangle in $D(X_m)$
$$ F[1] \to L\iota_m^\ast \iotam_\ast F \to F \to F[2].$$
\end{coro}
\begin{proof}
We have the adjunction morphism $L\iotam^\ast \iotam_\ast F \overset{\eta}{\to} F$.  Define $K$ by the exact triangle
 \[
 K \to L\iota_m^\ast \iotam_\ast F \overset{\eta}{\longrightarrow} F \to K[1].
 \]
Taking the long exact sequence of cohomology with respect to $\Ap_m$, we get $\HH^0 (K)=0$ (because $\HH^0 (\eta)$ is an isomorphism by Lemma \ref{cohomology-lemma5}), $\HH^{-1}(K) \cong \HH^{-1} L{\iota^\ast_m} \iotam_\ast F$, and $\HH^i (K)=0$ for all $i \neq 0, -1$.  Therefore, $K \cong \HH^{-1}(K)[1]$ and it remains to show that $\HH^{-1} L{\iota^\ast_m} \iotam_\ast F \cong F$ in $\Ap_m$.  To see this, consider the canonical map $\al : (\HH^{-1} L\iota^\ast_m \iotam_\ast F)[1] = \tau_{\Ap_m}^{\leq -1} L\iota^\ast_m \iotam_\ast F  \to L\iota^\ast_m \iotam_\ast F$.  Let $\al'$ denote the composition
\[
  \iotam_\ast (\HH^{-1} L\iota^\ast_m \iotam_\ast F) [1] \overset{\iotam_\ast \al}{\longrightarrow} \iotam_\ast L\iota^\ast_m \iotam_\ast F \cong \iotam_\ast F \oplus \iotam_\ast F[1] \to \iotam_\ast F[1]
\]
where the last map is projection onto the second factor.  By Lemma \ref{cohomology-lemma4}, it suffices to show that $\al'$ is an isomorphism.  And it is enough to show that $H^i (\al')$ is an isomorphism for $i=-2,-1$.  That $H^{-2}(\al')$ is an isomorphism is easy to see.  On the other hand, $H^{-1}(\al)$ is the canonical injection of the torsion part of $H^{-1}L\iota^\ast_m \iotam_\ast F$ into $H^{-1}L\iota^\ast_m \iotam_\ast F$ itself, with respect to the torsion pair $(\Coh_{\leq 1}(X_m), \Coh_{\geq 2}(X_m))$ in $\Coh(X_m)$.  Hence  $H^{-1}(\al')$ is also an isomorphism.
\end{proof}

\begin{lemma}\label{lemma-cohomology6}
Let $1 \leq m' < m$, and let $A \in \Ap_{m'}$.  Then all the $\HH^i_{\Ap_{m'}} L\iota^\ast_{m,m'} {\iota_{m,m'}}_\ast A$ for odd $i < 0$ are isomorphic.
\end{lemma}

\begin{proof}
For brevity, let us write $f$ to denote $\iota_{m,m'} : X_{m'} \hookrightarrow X_m$ just in this proof.  Then
\begin{align*}
 f_\ast Lf^\ast f_\ast A &\cong f_\ast A \Lo [ \cdots \to \OO_{X_m} \overset{\pi^{m'}}{\longrightarrow} \OO_{X_m} \overset{\pi^{m-m'}}{\longrightarrow} \OO_{X_m} \overset{\pi^{m'}}{\longrightarrow} \OO_{X_m} ] \\
&\cong f_\ast A \oplus \left( \bigoplus_{n<0, n \, \textnormal{odd}} [f_\ast A \overset{\pi^{m-m'}}{\longrightarrow} f_\ast A][n]\right).
\end{align*}

For any odd integer $i<0$, let $\al : \tau^{\leq i} Lf^\ast f_\ast A \to Lf^\ast f_\ast A$ denote the canonical map, and let $\al'$ denote the composition
\[
  f_\ast \tau^{\leq i} Lf^\ast f_\ast A \overset{f_\ast \al}{\longrightarrow} f_\ast Lf^\ast f_\ast A \overset{p}{\to} f_\ast T_i
\]
where the second map $p$ is projection onto $f_\ast T_i$, where $T_i := \bigoplus_{n \leq i, n\, \text{odd}} [A \overset{\pi^{m-m'}}{\longrightarrow} A][n]$, which is an object in $D^{\leq i}_{\Ap_{m'}}$.  If we can show that $\al'$ is an isomorphism, then by the $t$-exactness of $f_\ast$ and Lemma \ref{cohomology-lemma4}, we would have $\HH^n (\tau^{\leq i} Lf^\ast f_\ast A) \cong \HH^n (T_i)$ for all odd integers $n \leq i$, and then the lemma follows from the 2-periodicity of $T_i$.  To show that $\al'$ is an isomorphism, it is enough to demonstrate that $H^s(\al')$ is an isomorphism for every integer $s$.  Note that, for $s>i$, both $H^s(\tau^{\leq i} Lf^\ast f_\ast A)$ and $H^s(T_i)$ are zero.

At degree $i$, let $h_t, h_f$ denote the torsion and torsion-free parts of $H^i (Lf^\ast f_\ast A)$ with respect to the torsion pair $(\Coh_{\leq 1}, \Coh_{\geq 2})$ in $\Coh (X_{m'})$, respectively.  Then $H^i(f_\ast \al)=f_\ast H^i(\al)$ is the canonical injection of $f_\ast h_t$ into $H^i(f_\ast Lf^\ast f_\ast A)$, while $H^i(p)$ is the canonical projection onto $h_t$, and so $H^i(\al')$ is an isomorphism.  For $s<i$, both $H^s(f_\ast \al)$ and $H^s(p)$ are isomorphisms.  Therefore, $\al'$ is an isomorphism in the derived category.
\end{proof}

\subsection{Flat Families in $D^b(X_m)$ are in $\Ap_m$}

 We will now establish that, if $I$ is any flat family of objects in $\Ap$ over $\Spec R$, then for any $m\geq 1$ we have $L\iota_m^\ast I \in \Ap_m$.  In fact, we prove something slightly more general:

\begin{pro}\label{XkflatXmflat}
Given an object $E \in D^{\leq 0}_{\Coh (X_m)}(X_m)$, if $L\iota_{m,1}^\ast E \in \Ap (X_k)$ then $E \in \Ap_m$.
\end{pro}

This proposition will follow from the next three lemmas.

\vspace{5pt}

  For the next three lemmas, let us use the following notation: let $A$ be a finitely generated $k$-algebra that is an integral domain, and let $R$ be a DVR with uniformiser $\pi$.  Fix some positive integer $m>1$, and let $B := A \otimes_k R/(\pi^m)$.

Recall that the Krull dimension of a $B$-module $M$ is defined as $\dimension (B/\mbox{ann}(M))$, and this is how the dimension of a coherent sheaf is calculated locally.

\begin{lemma}\label{flatness-lemma1}
Let $$E := [\cdots \to E^{-1} \overset{\phi}{\to} E^0 \to 0]$$ be a chain complex (not necessarily bounded from below) of finite-rank free $B$-modules, such that $\dimension H^0(E\otimes_B R/\pi) < \dimension A$.  Then as $B$-modules,
\[
  \dimension H^0(E) \leq \dimension H^0(E\otimes_B R/\pi).
\]
\end{lemma}

On the derived category level, this lemma implies that, if $E$ is a flat family of objects in $\Ap$ over $\Spec R/\pi^m$ (so that $H^0(E \otimes_B R/\pi) \in \Coh_{\leq 1}(X_k)$), then $H^0(E) \in \Coh_{\leq 1}(X_m)$.

\begin{proof}
We can regard $H^0(E)= \frac{E^0}{\image \phi}$ and $H^0(E \otimes_B R/\pi)= \frac{E^0\otimes_B R/\pi}{\image (\phi \otimes_B R/\pi)}$ as $B$-modules.  Take any nonzero $g$ in $\mbox{ann}(H^0(E \otimes_B R/\pi))$, which is defined as the quotient ideal
$$(\mbox{im} (\phi\otimes_B R/\pi) : E^0\otimes_B R/\pi):= \{ x \in B : x(E^0 \otimes_B R/\pi ) \subseteq \mbox{im}(\phi \otimes_B R/\pi) \}.$$
Since $\dimension H^0(E\otimes_B R/\pi) < \dimension A$, we have $\mbox{ann}(H^0(E \otimes_B R/\pi)) \nsubseteq (\pi)$, and so we can find some $g$ in $\mbox{ann}(H^0(E\otimes_B R/\pi))$ such that $g \nequiv 0$ mod $\pi$.  We claim that there is an integer $q$ such that $g^q \in \mbox{ann}(H^0(E))$.

Since $E^0$ is a free $B$-module of rank $n$, we can fix a $B$-module isomorphism $E^0 \cong B^{\oplus (\mbox{rank}\, E^0)}$, and let $e_i = (0,\cdots,1,\cdots,0)$ be the standard basis element with $1$ in the $i$-th summand.  Under this isomorphism, for any $i$ we have $ge_i \in \image \phi$ mod $\pi$.  If we also fix an isomorphism $E^{-1} \cong B^{\oplus (\mbox{rank}\, E^{-1})}$, we can associate to $\phi$ a matrix $T$ over $B$.  Then there exists $a_i \in E^{-1}$ such that $Ta_i = g e_i$ mod $\pi$.

Now we invert $g$ (If $\nu$ is the valuation on $K$ associated to $R$, then $\nu (g)=0$, and so $g$ is not nilpotent.  So inverting  $g$ does not kill all of $B$.)   Then we have that $g^{-1}Ta_i = e_i$ mod $\pi$ for each $i$, and so for each $i$, we have $e_i = g^{-1}Ta_i$ mod $\pi$, so $e_i = g^{-1}Ta_i + \pi b$ in $E^0$ for some $b$.

Therefore, after inverting $g$, each $e_i$ lies in $\image \phi$ (using the equations we just obtained, the part that is a multiple of $\pi$ can be written as the sum of something else in $\image \phi$ plus a multiple of $\pi^2$, etc.).  That is, the localisation $(H^0(E))_{g}= 0$, meaning there is some positive integer $q$ such that $g^q H^0(E^0)=0$, i.e.\ $g^q \in \mbox{ann} (H^0(E))$.

Now suppose $\mbox{ann}(H^0(E\otimes_B R/\pi))$ is generated by $\pi, g_1,\cdots, g_l \in B$ where $g_j \neq 0$ mod $\pi$ for all $1 \leq j \leq l$, and that $q_j$ are positive integers such that $g_j^{q_j} \in \mbox{ann} (H^0(E))$.  Then
\begin{align*}
  \dimension H^0(E\otimes_B R/\pi) &= \dimension B/\mbox{ann} H^0(E \otimes_B R/\pi) \\
&= \dimension B/(g_1,\cdots, g_l) \\
&= \dimension B/( g_1^{q_1} ,\cdots, g_l^{q_l} ) \\
&\geq \dimension B/\mbox{ann}(H^0(E)) \\
&= \dimension H^0(E)
\end{align*}
as wanted.
\end{proof}

\begin{lemma}\label{flatness-lemma2}
If $E \in D^-(X_m)$ is such that $H^i(E)=0$ for all $i>0$, then $\dimension H^0(E) \geq \dimension H^0(E\otimes_B R/\pi)$.
\end{lemma}

\begin{proof}
Since $E$ has zero cohomology at all positive degrees $i$, we can trim $E$ to a complex of the form $[\cdots \to E^{-1} \overset{\phi}{\to} E^0 \to 0 \cdots]$ where each $E^i$ is at degree $i$ and is locally free.  We have $H^0(E) = \frac{E^0}{\image \phi}$ and $H^0(E\otimes_B R/\pi) = \frac{E^0\otimes_B R/\pi}{\image (\phi\otimes_B R/\pi)}$,
from which we see $\mbox{ann}\, H^0(E) \subseteq \mbox{ann} (H^0(E\otimes_B R/\pi))$, and so $\dimension H^0(E) \geq \dimension H^0(E\otimes_B R/\pi)$ as wanted.
\end{proof}

Lemmas \ref{flatness-lemma1} and \ref{flatness-lemma2} together imply: given any $E \in D^-(X_m)$ with $H^i(E)=0$ for $i>0$, and such that $\dimension H^0(E\otimes_B R/\pi)< \dimension A$, we have $\dimension H^0(E) = \dimension H^0(E\otimes_B R/\pi)$ as $B$-modules.  Hence $H^0(E) \in \Coh_{\leq 1}(X_m)$ iff $H^0(E\otimes_B R/\pi) \in\Coh_{\leq 1}(X_k)$.  Consequently, if $I \in D^b(X_R)$ is a flat family of objects in $\Ap$ over $\Spec R$, then $H^0(L\iota_m^\ast I) \in \Coh_{\leq 1}(X_m)$ for each $m \geq 1$.

\vspace{5pt}

The next lemma will imply that, given a flat family $I \in D^b(X_R)$ of objects in $\Ap$ over $\Spec R$, we will have $H^{-1}(L\iota_m^\ast I) \in \Coh_{\geq 2}(X_m)$ for each $m \geq 1$.  Or, even, given a flat family $I \in D^b(X_m)$ of objects in $\Ap$ over $\Spec R/\pi^m$, we will have $H^{-1}(L\iota_{m,m'}^\ast I) \in \Coh_{\geq 2}(X_{m'})$ for any $1 \leq m' < m$.

\begin{lemma}\label{flatness-lemma3}
Let \[E = [\cdots \to E^{-3} \overset{\nu}{\to} E^{-2} \overset{\psi}{\to} E^{-1} \overset{\phi}{\to} E^0 \to 0 \to \cdots] \in D^-(X_m)\] be a chain complex of finite-rank free $B$-modules such that $H^{-2}(E\otimes_B R/\pi)=0$ (and hence $H^{-2}(E)=0$).  Then if $H^{-1}(E)$ has a nonzero submodule of dimension $d < \dimension A$, then $H^{-1}(E\otimes_B B/\pi^{m-1})$ also has a nonzero submodule of dimension $\leq d$.
\end{lemma}

\noindent\textit{Remark.} That the vanishing of $H^{-1}(E\otimes_B k)$ implies the vanishing of $H^{-1}(E)$ itself follows from \cite[Lemma 2.1.4]{Lieblich}.

\begin{proof}
Suppose $0 \neq T \subseteq H^{-1}(E)$ is a $d$-dimensional $B$-submodule, where $d<\dimension A$.  We want to produce a nonzero submodule $T' \subseteq H^{-1}(E \otimes_B B/\pi^{m-1})$ of dimension at most $d$.

Let $\theta$ denote the map
\[
\theta : H^{-1}(E) = \frac{\kernel \phi}{\image \psi} \to \frac{\kernel (\phi\otimes_B B/\pi^{m-1})}{\image (\psi \otimes_B B/\pi^{m-1})} = H^{-1}(E \otimes B/\pi^{m-1}).
\]
Suppose $g_1, \cdots, g_r$ in $\kernel \phi$ generate $T$ as a $B$-module.  Let $\bar{g_i}$ denote the image of $g_i$ in $H^{-1}(E)$.

Observe that, for the $B$-modules $T=(\bar{g_1},\cdots,\bar{g_r})$ and $\theta(T)=(\theta (\bar{g_1}),\cdots,\theta (\bar{g_r}))$, we have $\mbox{ann}(T) \subseteq \mbox{ann} (\theta (T))$. Therefore,
\[
\dimension T = \dimension B/\mbox{ann}(T) \geq \dimension B/\mbox{ann}(\theta(T)) = \dimension \theta (T).
\]
And so if $\theta (T)$ is nonzero, it would be a $B$-submodule of $H^{-1}(E \otimes B/\pi^{m-1})$ of dimension $\leq d$.

In the event that $\theta (T)=0$, we produce a somewhat different nonzero $B$-submodule of $H^{-1}(E \otimes B/\pi^{m-1})$ of dimension $\leq d$, as follows.  Assume from now on that $\theta (T)=0$.  This means $g_i \in (\image \psi) + \pi^{m-1}E^{-1}$ for all $i$.  So we might as well assume that each $g_i$ is a multiple of $\pi$ (since the $g_i$ generate a nonzero submodule in $H^{-1}(E)$, we can replace them by their residue modulo $\image \psi$), say $g_i=\pi h_i$, and that $g_i \notin \image \psi$.

Also, $\theta(h_i) \in (\kernel \phi\otimes_B B/\pi^{m-1})$ for all $i$.  This is because $g_i=\pi h_i \in \kernel \phi$ means $0=\phi (\pi h_i)=\pi \phi (h_i)$ in $E^0$, and so $\phi (h_i) \in \pi^{m-1}E^{-1}$.  It follows that $\theta(h_i) \in \kernel (\phi\otimes_B B/\pi^{m-1})$, so $T':=(\theta(h_1),\cdots,\theta(h_r))$ can be considered as a submodule of $H^{-1}(E\otimes B/\pi^{m-1})$.

Next, notice that $T'$ is a nonzero submodule of $H^{-1}(E\otimes_B B/(\pi^{m-1}))$: if $\theta (h_i)=0$ for all $i$, that means that $h_i$ represents an element in $(\image \psi\otimes_B B/\pi^{m-1}) \subset E^{-1} \otimes_B B/\pi^{m-1}$ modulo $\pi^{m-1}E^{-1}$, and so $g_i=\pi h_i \in \image \psi$, contradicting our assumption.

Fix an $i$.  Considering each principal $B$-module $(\pi h_i)$ as a $B$-submodule of $H^{-1}(E)$ and $(\theta(h_i))$ as a $B$-submodule of $H^{-1}(E \otimes B/\pi^{m-1})$, we want to show that $\mbox{ann}(\pi h_i) \subseteq \mbox{ann} (\theta (h_i))$.  To achieve this, take any $g \in \mbox{ann}(\pi h_i)$.  Fix an isomorphism $E^{-2} \cong B^{\oplus (\mbox{rank}\, E^{-2})}$, and let $S$ be the matrix over $B$ associated to the differential $\psi$.  Write $S$ as $S=\oplus_{j=0}^{m-1} \pi^j S_j$, where the $S_j$ are matrices over $A$.

Since $g$ annihilates $\pi h_i$, it means that $g\pi h_i \in \image \psi$, say $g\pi h_i = Sf$ for some $f \in E^{-2}$.  Then $Sf=0$ mod $\pi$.  If we can show that $g \pi h_i = \pi S v'$ for some $v' \in E^{-2}$, then we have that $g h_i  \in \image (\psi\otimes_B B/\pi^{m-1})$, and $g \in\mbox{ann}(\theta (h_i))$ would follow.

Now, that $Sf=0$ mod $\pi$ means $f \in \kernel (\psi \otimes_B R/\pi) = \image (\nu \otimes_B R/\pi)$.  And so $f = Uv$ mod $\pi$ for some $v \in E^{-3}$, i.e.\ $f = Uv + \pi b$ for some $b \in E^{-2}$.  Then $g\pi h_i = Sf = SUv + \pi Sb = 0 + \pi Sb$,
  and so $Sb = gh_i$.  So $g$ kills $\theta (h_i)$ mod $\pi^{m-1}$.  So we have shown that $\mbox{ann}(\pi h_i) \subseteq \mbox{ann}(\theta (h_i))$.

To finish off, we note
\begin{gather*}
  \mbox{ann} (T) = \bigcap_i \mbox{ann} (\pi h_i) \subseteq \bigcap_i \mbox{ann} (\theta (h_i)) = \mbox{ann} (\theta (h_1),\cdots,\theta (h_r))
\end{gather*}
so $\dimension T \geq \dimension (\theta (h_1),\cdots,\theta (h_r)) = \dimension \theta (T')$.

Hence we have produced a nonzero submodule of $H^{-1}(E \otimes B/\pi^{m-1})$ of dimension $\leq d$, and we are done.
\end{proof}

\paragraph{Remark.}  Note that, all the results in section \ref{section-ppp} and Proposition \ref{XkflatXmflat} hold if we replace $\Ap_m$ with $\Coh (X_m)$.  It is also hoped that the techniques developed here apply to any situation where we have:
\begin{enumerate}
\item[(a)] a collection of abelian categories $\{A_m\}_{m \geq 1}$, and a t-structure on each derived category $D(A_m)$ obtained by tilting the standard t-structure on $D(A_m)$;
\item[(b)] `pushforward functors' ${\iota_{m,m'}}_\ast : A_{m'} \to A_m$ and `pullback functors' ${\iota^\ast_{m,m'}} : A_{m} \to A_{m'}$ for any $1 \leq m' < m$;
\item[(c)] with respect to the non-standard t-structures in (a), the derived pushforward functors are t-exact  and the derived pullback functors are right t-exact.
\end{enumerate}

\end{document}